\documentclass[11pt]{amsart}

\newcommand{\preprint}[1]{}

\newcommand{\WithLieblich}[1]{#1}
\newcommand{\WithoutLieblich}[1]{}

\newcommand{\hide}[1]{}

\usepackage{amssymb}
\usepackage{amsbsy}
\usepackage{amscd}
\usepackage{amsmath}
\usepackage{amsthm}

\numberwithin{equation}{section}

\theoremstyle{plain}
\newtheorem{thm}{Theorem}[section]
\newtheorem{prop}[thm]{Proposition}
\newtheorem{claim}[thm]{Claim}
\newtheorem{problem}[thm]{Problem}

\newtheorem{cor}[thm]{Corollary}
\newtheorem{lem}[thm]{Lemma}

\theoremstyle{definition}
\newtheorem{defi}[thm]{Definition}

\newtheorem{rem}[thm]{Remark}
\newtheorem{caution}[thm]{Caution}

\newtheorem{construction}[thm]{Construction}

\newcommand{\A}{{\mathcal A}}

\newcommand{\C}{{\mathcal C}}

\newcommand{\E}{{\mathcal E}}
\newcommand{\F}{{\mathcal F}}
\newcommand{\G}{{\mathcal G}}
\renewcommand{\H}{{\mathcal H}}

\newcommand{\M}{{\mathcal M}}

\renewcommand{\P}{{\mathcal P}}
\newcommand{\PP}{{\mathbb P}}
\newcommand{\Q}{{\mathcal Q}}
\newcommand{\R}{{\mathcal R}}
\renewcommand{\S}{{\mathcal S}}

\newcommand{\U}{{\mathcal U}}

\newcommand{\X}{{\mathcal X}}
\newcommand{\Y}{{\mathcal Y}}
\newcommand{\Z}{{\mathcal Z}}
\newcommand{\RealNumbers}{{\mathbb R}}
\newcommand{\Integers}{{\mathbb Z}}
\newcommand{\ComplexNumbers}{{\mathbb C}}
\newcommand{\RationalNumbers}{{\mathbb Q}}
\newcommand{\LieAlg}[1]{{\mathfrak #1}}

\newcommand{\monrep}{{mon}}
\newcommand{\monclass}{{\overline{mon}}}

\newcommand{\LongIsomRightArrow}{\stackrel{\cong}{\longrightarrow}}
\newcommand{\RightArrowOf}[1]{\stackrel{#1}{\rightarrow}}

\newcommand{\LongRightArrowOf}[1]{\stackrel{#1}{\longrightarrow}}

\newcommand{\Lotimes}{\stackrel{L}{\otimes}}
\newcommand{\StructureSheaf}[1]{{\mathcal O}_{#1}}
\newcommand{\EndProof}{\hfill  $\Box$}
\newcommand{\restricted}[2]{#1_{\mid_{#2}}}

\newcommand{\rank}{{\rm rank}}
\newcommand{\coker}{{\rm coker}}
\newcommand{\slope}{{\rm slope}}
\renewcommand{\Im}{{\rm Im}}
\newcommand{\Pic}{{\rm Pic}}
\newcommand{\Sym}{{\rm Sym}}
\newcommand{\Ext}{{\rm Ext}}
\newcommand{\Tor}{{\mathcal T}or}
\newcommand{\Hom}{{\rm Hom}}
\newcommand{\Aut}{{\rm Aut}}
\newcommand{\End}{{\rm End}}

\newcommand{\Abs}[1]{\left|\!#1\!\right|}

\newcommand{\SheafHom}{{\mathcal H}om}
\newcommand{\SheafEnd}{{\mathcal E}nd}

\newcommand{\SheafExt}{{\mathcal E}xt}

\newcommand{\DerivedOtimes}{\stackrel{L}{\otimes}}

\newcommand{\Wedge}[1]{\stackrel{#1}{\wedge}}

\newcommand{\Choose}[2]{\left(\!\!\begin{array}{c}#1\\#2\end{array}\!\!\right)}

\begin{document}
\title[The Beauville-Bogomolov class as a characteristic class]
{The Beauville-Bogomolov class as a characteristic class}
\author{Eyal Markman}
\address{Department of Mathematics and Statistics, 
University of Massachusetts, Amherst, MA 01003, USA}
\email{markman@math.umass.edu}

\begin{abstract}
Let $X$ be any compact K\"{a}hler manifold deformation equivalent to the 
Hilbert scheme of length $n$ subschemes on a $K3$ surface, $n\geq 2$.
We construct over $X\times X$ a rank $2n-2$  reflexive twisted coherent sheaf $E$, which is  locally free away from the diagonal. The characteristic classes $\kappa_i(E)\in H^{i,i}(X\times X,\RationalNumbers)$ of $E$ are invariant under the diagonal action of an index $2$ subgroup of the monodromy group of $X$. 
Given a point $x\in X$, the restriction $E_x$ of $E$ to $\{x\}\times X$ has the following properties.
\begin{enumerate}
\item
The characteristic class $\kappa_i(E_x)\in H^{i,i}(X,\RationalNumbers)$ 
can not be expressed as a polynomial in classes of 
lower degree, if $2\leq i\leq n/2$.
\item
The Beauville-Bogomolov class is equal to $c_2(TX)+2\kappa_2(E_x)$.
\end{enumerate}
\end{abstract}
\maketitle
\tableofcontents 

%
\section{Introduction}
\label{sec-introduction}
%
\subsection{The main results} 
An {\em irreducible holomorphic symplectic manifold} is a simply connected
compact K\"{a}hler manifold $X$, such that $H^0(X,\Omega^2_X)$ is generated
by an everywhere non-degenerate holomorphic two-form.
Let $S$ be a smooth K\"{a}hler $K3$ surface and
$S^{[n]}$ the Hilbert scheme of length $n$ zero dimensional subschemes of $S$.
$S^{[n]}$ is an irreducible holomorphic symplectic 
manifold \cite{beauville-varieties-with-zero-c-1}. 
An irreducible holomorphic symplectic manifold $X$ is said to be of 
{\em $K3^{[n]}$-type}, if $X$ is deformation equivalent to $S^{[n]}$, for a $K3$ surface $S$.
The moduli space of manifolds of $K3^{[n]}$-type 
is $21$-dimensional, if $n\geq 2$. 
In particular, a generic manifold of $K3^{[n]}$-type  is not
the Hilbert scheme of any $K3$ surface.

Let $Y$ be a compact K\"{a}hler manifold. A Hodge class $\alpha\in H^{i,i}(Y,\RationalNumbers)$ is said to be {\em analytic}, 
if $\alpha$ belongs to the subring of $H^*(Y,\RationalNumbers)$ generated by the Chern classes of coherent analytic sheaves on $Y$. When $Y$ is projective, a class is analytic if and only if it is algebraic. The aim of this paper is to prove that certain interesting Hodge classes on the product $X\times X$ of every manifold of $K3^{[n]}$-type, $n\geq 2$, are analytic. We define these Hodge classes next via parallel transport of monodromy invariant Hodge classes on $S^{[n]}\times S^{[n]}$, where $S$ is a $K3$ surface. 

\begin{defi}
\label{def-monodromy}
{\rm
Let $X$ be an irreducible holomorphic symplectic manifold. 
An automorphism $g$ of the cohomology ring 
$H^*(X,\Integers)$ is called a {\em monodromy operator}, 
if there exists a 
family $\X \rightarrow B$ (which may depend on $g$) 
of irreducible holomorphic symplectic manifolds, having $X$ as a fiber
over a point $b_0\in B$, 
and such that $g$ belongs to the image of $\pi_1(B,b_0)$ under
the monodromy representation. 
The {\em monodromy group} $Mon(X)$ of $X$ is the subgroup 
of $GL(H^*(X,\Integers))$ generated by all the monodromy operators. 
}
\end{defi}


Let $\E$ be the ideal sheaf of the universal subscheme in $S\times S^{[n]}$. Let $\pi_{ij}$ be the projection
from $S^{[n]}\times S\times S^{[n]}$ onto the product of the $i$-th and $j$-th factors. Let 
\begin{equation}
\label{eq-E}
E \ \ := \ \ \SheafExt^1_{\pi_{13}}\left(\pi_{12}^*\E,\pi_{23}^*\E\right)
\end{equation}
be the relative extension sheaf over $S^{[n]}\times S^{[n]}$. 

The cohomology group $H^2(X,\Integers)$, of an irreducible holomorphic 
symplectic manifold $X$, admits a canonical, symmetric, non-degenerate, 
and primitive bilinear pairing  called the Beauville-Bogomolov pairing
\cite{beauville-varieties-with-zero-c-1}.
The discriminant group $H^2(S^{[n]},\Integers)^*/H^2(S^{[n]},\Integers)$ is cyclic of order $2n-2$ and $Mon(S^{[n]})$
acts on it by $\pm 1$, by \cite[Lemma 4.2]{markman-integral-constraints}.
We get a group homomorphism 
\begin{equation}
\label{eq-character-rho}
\rho:Mon(S^{[n]})\rightarrow \mu_2, 
\end{equation}
which surjects onto the multiplicative group of two elements if $n\geq 3$, and is trivial if $n=2$.

\begin{prop}
\label{prop-definition-of-kappa-E}
\begin{enumerate}
\item
(Proposition \ref{prop-V}) The sheaf $E$ is reflexive of rank $2n-2$ and is locally free away from the diagonal.
\item
\label{prop-item-kappa-E-is-monodromy-invariant}
(Proposition \ref{prop-kappa-E-is-mon-invariant})
Set $\kappa(E):=ch(E)\cup \exp\left(\frac{-c_1(E)}{2n-2}\right)$ and let $\kappa_i(E)$ be its graded summand in 
$H^{2i}(S^{[n]}\times S^{[n]},\RationalNumbers)$. 
The subspace
$\mbox{span}_\RationalNumbers\{\kappa_i(E)\}$
is invariant under the diagonal $Mon(S^{[n]})$-action. 
When $i$ is even the class $\kappa_i(E)$ is $Mon(S^{[n]})$-invariant. When $i$ is odd 
$Mon(S^{[n]})$ acts on $\mbox{span}_\RationalNumbers\{\kappa_i(E)\}$ via the character $\rho$.
\end{enumerate}
\end{prop}

\begin{rem}
The sheaf $E^*:=\SheafHom(E,\StructureSheaf{S^{[n]}\times S^{[n]}})$ dual to $E$ 
and the derived dual object 
$R\SheafHom(E,\StructureSheaf{S^{[n]}\times S^{[n]}})$ have the same $\kappa$-class, 
by Proposition \ref{prop-V}(\ref{lemma-item-SheafExt-1-is-reflexive}).
The two objects are not isomorphic, but have the same class in the algebraic $K$-ring, by 
\cite[Lemma 4.3]{markman-mehrotra-integral-transforms}. 
The above Proposition thus states that for $g \in Mon(S^{[n]})$,
\[
g(\kappa(E))=\left\{
\begin{array}{ccl}
\kappa(E), & \mbox{if} & \rho(g)=1,
\\
\kappa(E^*), & \mbox{if} & \rho(g)=-1.
\end{array}
\right.
\]
Note also that $\kappa(E)=\kappa(E^*)$ if $n=2$, since in that case the rank of $E$ is equal to $2$ and so $E^*$ is isomorphic to $E\otimes\det(E)^*$.
\end{rem}

Parallel transport of a class $\alpha$ in $H^{2i}(S^{[n]}\times S^{[n]},\RationalNumbers)$, 
which is invariant under the diagonal action of $Mon(S^{[n]})$,
defines a class $\alpha_X\in H^{2i}(X\times X,\RationalNumbers)$ for any  $X$ of $K3^{[n]}$-type. 
More generally, if $\mbox{span}_\RationalNumbers\{\alpha\}$
is a one-dimensional $Mon(S^{[n]})$-representation, then 
we get a well defined unordered pair $\pm \alpha_X$ of
a class and its negative. Such a class $\alpha_X$ is of Hodge type $(i,i)$,
by  Lemma \ref{lemma-Mon-invariant-classes-are-of-Hodge-type}.
Denote by
\[
\pm \kappa_i(X\times X) \in H^{2i}(X\times X,\RationalNumbers)
\]
the pair of Hodge classes on a manifold of $K3^{[n]}$-type obtained from the classes $\kappa_i(E)$ via parallel transport. 
We would like to stress that the monodromy invariance of the classes $\kappa_i(E)$ in Proposition 
\ref{prop-definition-of-kappa-E} (\ref{prop-item-kappa-E-is-monodromy-invariant}) is an easy consequence of a monodromy equivariance property of the universal sheaf $\E$ over $S\times S^{[n]}$ proven in
\cite{markman-monodromy-I,markman-integral-constraints}. 
The monodromy invariance of $\kappa_i(E)$  motivated the current work and it is a crucial ingredient in the proof of the main result stated below.

Our main result stated next implies that the Hodge classes $\kappa_i(X\times X)$ are analytic.
Given a coherent sheaf $F$ of rank $r>0$ over a complex manifold $Y$ twisted by some Brauer class, we get the untwisted object 
$F^{\otimes r}\otimes \det(F)^{-1}$ in the derived category of $Y$. Denote the $r$-th root of the Chern character of this object by 
$\kappa(F)$ and let $\kappa_i(F)$ be its graded summand in $H^{2i}(Y,\RationalNumbers)$. When $F$ is untwisted, this new definition of $\kappa(F)$ agrees with the one in Proposition \ref{prop-definition-of-kappa-E}. Details are provided in Section \ref{sec-twisted-sheaves}. 

\begin{thm}
\label{thm-main-introduction}
Let $X$ be a manifold of $K3^{[n]}$-type, $n\geq 2$. There exists over $X\times X$ a rank $2n-2$ reflexive twisted 
coherent sheaf $F$, which is locally free away from the diagonal and satisfies
$\kappa_i(F)=\kappa_i(X\times X)$, for $2\leq i\leq 2n-1$.
\end{thm}

The above statement is proved in Section \ref{sec-proof-of-deformability}. 
The class $\kappa_i(X\times X)$ is well defined above when $i$ is even. 
When $i$ is odd it is defined only up to sign. However, 
$\kappa_i(F^*)=(-1)^i\kappa_i(F)$, for $i$ in the above range, since $F$ is locally free away from the diagonal, so the existence of $F$
satisfying the equality $\kappa_i(F)=\kappa_i(X\times X)$ follows in spite of the sign ambiguity.
The sheaf $F$ is constructed as a deformation of the sheaf 
$E$ given in Equation (\ref{eq-E}). The fact that the sheaf $E$ deforms from $S^{[n]}\times S^{[n]}$ to $X\times X$ is established as follows. 
One first uses standard results in the theory of moduli spaces of sheaves on $K3$ surfaces to deform $E$ to a reflexive twisted sheaf $E'$ 
over the self product $\M\times \M$ of a moduli space $\M$ of rank $r:=2n-2$ stable sheaves over a $K3$ surface $S'$ with a cyclic Picard group generated by an ample line bundle of degree $2r^2+r$. The sheaf $E'$ is defined in terms of the universal twisted sheaf over $S'\times \M$ as in Equation (\ref{eq-E}). The sheaf $E'$ is maximally twisted, the order of its Brauer class is equal to its rank.
This fact is used to prove the slope-polystability of $\SheafEnd(E')$ with respect to every K\"{a}hler class on the product. The slope-polystability, coupled with the invariance of $c_2(\SheafEnd(E'))$ with respect to the diagonal monodromy action, enables us to use a theorem of Verbitsky to deform $E'$ to a sheaf $F$ over $X\times X$, for every $X$ of $K3^{[n]}$-type.

Given a point $x\in X$, denote by $\kappa_i(X)$ the restriction of $\kappa_i(X\times X)$ to $\{x\}\times X$. 
Theorem \ref{thm-main-introduction} yields an 
expression of the Beauville-Bogomolov pairing $q\in \Sym^2H^2(X,\Integers)^*$ 
in terms of characteristic classes, 
for $X$ of $K3^{[n]}$-type, $n\geq 2$, 
by the following Lemma. 
The inverse of $q$ is a class in $\Sym^2H^2(X,\RationalNumbers)$, and we 
denote by $q^{-1}$ its image  in 
$H^4(X,\RationalNumbers)$ as well. 

\begin{lem} 
\label{lem-BB-form-is-motivic}
The following equation holds in $H^4(X,\RationalNumbers)$, 
for any  $X$ of $K3^{[n]}$-type, $n\geq 2$.
\begin{equation}
\label{eq-relation-between-kappa-2-c-2-and-BB}
q^{-1} \ \ \ = \ \ \ c_2(TX) + 2\kappa_2(X).
\end{equation}
The dimension of the subspace\footnote{
When $n=3$, the relation
$4q^{-1}=3c_2(T\M)$ holds as well. It follows from 
Chern numbers calculations,
by comparing two formulas for  the Euler characteristic 
$\chi(S^{[n]},L)$ of a line bundle $L$ on $S^{[n]}$.
One as a binomial coefficient 
$
\chi(S^{[n]},L) = \Choose{\frac{q(c_1(L),c_1(L))}{2}+n+1}{n}
$
\cite{ellingsrud-et-al}, the other provided by Hirzebruch-Riemann-Roch.
} 
$\mbox{span}\{q^{-1},c_2(TX),\kappa_2(X)\}$
is $2$, for $n\geq 4$, and $1$, for $n=2, 3$. 
\end{lem}

The Lemma is proven in Section 
\ref{sec-proof-of-relation-between-kappa-2-c-2-and-BB}.
More generally, the class $\kappa_i(X)$ is non-trivial; it 
can not be expressed as a polynomial in classes of
degree less than $2i$, if $i \leq \frac{n}{2}$
\cite[Lemma 10]{markman-diagonal}. In particular, the classes $\kappa_i(X\times X)$ are non trivial for $i$ in that range. 
In contrast, the odd Chern classes $c_{2k+1}(TX)$ vanish, since $TX$ is 
a holomorphic symplectic vector bundle.

In a separate paper with F. Charles the sheaf $F$ of Theorem \ref{thm-main-introduction} is used to prove the Standard Conjectures for $X$, whenever $X$ is a projective manifold of $K3^{[n]}$-type \cite{charles-markman}. 
In a separate paper with S. Mehrotra the sheaf $F$ is used to associate to any manifold of $K3^{[n]}$-type $X$ a pre-triangulated $K3$ category, yielding non-commutative deformations of the derived categories of coherent sheaves on $K3$ surfaces over the $21$-dimensional global moduli space of such $X$ \cite{markman-mehrotra-integral-transforms}. M. Shen and C. Vial used the sheaf $F$  of Theorem \ref{thm-main-introduction} in order to study a decomposition of the Chow ring of manifolds of $K3^{[2]}$-type \cite{shen-vial}. Addington studied in \cite{addington} the Fourier-Mukai transform with kernel the sheaf $F$ of Theorem \ref{thm-main-introduction} when $X$ is of $K3^{[2]}$-type.
He proved that this Fourier-Mukai transform is an auto-equivalence of the derived categories of $X$ in two cases: when $X$ is a Hilbert scheme $S^{[2]}$ of a $K3$ surface $S$, and when $X$ is the Fano variety of lines on a cubic fourfold. The Fourier-Mukai transform with respect to the sheaf $F$ of Theorem \ref{thm-main-introduction} is expected to induce an equivalence for a generic  $X$ of $K3^{[2]}$-type. Addington's construction in \cite{addington} produces derived auto-equivalences for $S^{[n]}$, $n>2$, 
and these are expected to deform as well due to the deformability Theorem  \ref{thm-main-introduction}.
\subsection{Notation}

Let $f:X\rightarrow Y$ be a proper morphism of complex manifolds
or smooth quasi-projective varieties. We denote by 
$f_*$ the push-forward  of coherent sheaves,
as well as the Gysin homomorphism in singular cohomology, 
while $f_!$ is the Gysin homomorphism in $K$-theory 
(algebraic, holomorphic \cite{grr-complex}, or topological). We let 
$K_{top}X$ be the
Grothendieck $K$-ring of equivalence classes of formal sums of 
topological vector bundles over $X$. 
\hide{
When $X$ is non-algebraic, a coherent sheaf 
may not have a locally free resolution, and
we denote by $K_0^{hol}(X)$ the Grothendieck $K$-homology group of 
coherent sheaves. The Gysin homomorphism, in the holomorphic category, 
is the homomorphism
\[
f_! \ : \ K_0^{hol}(X) \ \ \ \longrightarrow \ \ \ K_0^{hol}(Y)
\]
sending a coherent sheaf to the class of the alternating sum of its higher 
direct images. The Grothendieck-Riemann-Roch Theorem
is known in this generality \cite{grr-complex}.
When $X$ is smooth and quasi-projective, $K_0^{hol}(X)$, $K_{hol}(X)$, 
are naturally isomorphic. 
}

The pullback homomorphism is denoted by $f^*$
for coherent sheaves and in singular cohomology, 
while $f^!$ is the pull back in $K$-theory. 
Given a class $\alpha$ in $H^{even}(X)$, we denote by
$\alpha_i$ the graded summand in $H^{2i}(X)$. Given a class $y$ in the the $K$-ring, we denote its dual by $\alpha^\vee$.
Given a coherent sheaf $F$, we let $F^*:=\SheafHom(F,\StructureSheaf{X})$ be the dual sheaf, 
while $F^\vee$ denotes the dual object $R\SheafHom(F,\StructureSheaf{X})$ in 
the derived category of coherent sheaves.

The Chern character $ch(F)$ of a coherent analytic sheaf $F$ on a complex manifold $X$ is defined in $H^*(X,\RationalNumbers)$, using real analytic  resolutions via complex vector bundles, 
as in \cite{atiyah-hirzebruch}. 
An alternative definition in Deligne cohomology is given in \cite{grivaux} and 
the two definitions agree under the natural map from Deligne cohomology to $H^*(X,\RationalNumbers)$ \cite[Cor. 1]{grivaux}.
A definition of $ch(F)$ in the Hodge algebra $\oplus_{p=0}^{\dim(X)} H^p(X,\Omega^p_X)$, using the trace of the Atiyah class,
is given in \cite{grr-complex}, and a refinement due to H. I. Green in de Rham cohomology $H^*(X,\ComplexNumbers)$ is given in \cite{toledo-tong}. 
The Chern classes of coherent analytic sheaves are defined in $H^*(X,\RationalNumbers)$ in references \cite{atiyah-hirzebruch,grivaux} and the usual formulas relating them to the graded summands of the Chern character hold. 
When $X$ is a compact K\"{a}hler manifold all four definitions of the Chern character agree under the natural maps 
to de Rham cohomology  $H^*(X,\ComplexNumbers)$ and the Chern classes of a coherent analytic sheaf are Hodge classes 
\cite[Green's Theorem 2]{toledo-tong}.

Given a \v{C}ech $2$-cocycle $\theta$ of $\StructureSheaf{X}^*$
on a complex variety $X$, we define the notion of a
$\theta$-twisted coherent sheaf in Definition \ref{def-twisted-sheaves}.
A (coherent) sheaf will always mean an {\em untwisted} (coherent) sheaf,
unless we explicitly mention that it is twisted. 

%
\section{Characteristic classes of projective bundles and twisted sheaves}
\label{sec-characteristic-classes-of-projective-bundles}
Let $Y$ be a topological space and $y$ a class in the ring
$K_{top}Y$ generated by classes of complex vector bundles over $Y$.
Assume that the rank $r$ of $y$ is non-zero. Set
\begin{equation}
\label{eq-kappa}
\kappa(y) \ \ \ := \ \ \ ch(y)\cup \exp(-c_1(y)/r),
\end{equation}
and let $\kappa_i(y)$ be the summand of $\kappa(y)$ in 
$H^{2i}(Y,\RationalNumbers)$. 
In terms of the Chern roots $y_j$, we have
$ch_i(E)=\sum_{j=1}^r\frac{y_j^i}{i!}$, $c_1(E)=\sum_{j=1}^ry_j$,
and 
\[
\kappa_i(y) \ \ \ = \ \ \ 
\sum_{j=1}^r
\frac{\left[y_j-\left(\frac{\sum_{k=1}^ry_k}{r}\right)\right]^i}{i!}.
\]

The characteristic class $\kappa$ is multiplicative,
$\kappa(y_1\otimes y_2)=\kappa(y_1)\cup \kappa(y_2)$, and
$\kappa([L])=1$, for any line bundle $L$. 
Given a vector bundle $E$ over $Y$, 
the equality $\kappa(E)=\kappa(E\otimes L)$ thus holds, 
for any line bundle $L$. 
Note the equalities 
\begin{eqnarray*}
\kappa_i(y^\vee) & = & (-1)^i\kappa_i(y),
\\
\kappa(-y) & = & -\kappa(y).
\end{eqnarray*}

\hide{
\begin{lem}
\label{lemma-Grothendieck-Riemann-Roch}
Let $f:X\rightarrow Y$ be a regular bimeromorphic morphism of smooth 
compact complex manifolds and $E\in K_0^{hol}(X)$ a class of rank $r>0$.
Assume that $c_1(E)$ belongs to $f^*H^2(Y,\Integers)$. Then
\begin{equation}
\label{eq-Grothendieck-Riemann-Roch}
f_*(\kappa(E)\cdot td_{f}) \ \ \ = \ \ \
\kappa(f_!E).
\end{equation}
\end{lem}

\begin{proof}
Choose $\ell\in H^2(Y,\Integers)$, such that $c_1(E)=f^*\ell$. 
Then
\[
f_*(\kappa(E)\cdot td_{f}) = f_*(ch(E)\cdot td_{f})\cdot \exp(-\ell/r) =
ch(f_!(E))\exp(-\ell/r),
\]
by the Projection formula and Grothendieck-Riemann-Roch 
(extended to the complex analytic case \cite{grr-complex}). 
Now
\[
c_1(f_!E) = [f_*(ch(E)\cdot td_f)]_1=f_*[r\cdot (td_f)_1+c_1(E)] =
rc_1(f_!\StructureSheaf{X})+\ell.
\]
Equality (\ref{eq-Grothendieck-Riemann-Roch}) follows from the vanishing
$c_1(f_!\StructureSheaf{X})=0$. The latter vanishing follows from 
Zariski's Main Theorem, 
which implies that $f$ is an isomorphism, away from 
a subset of $Y$ of codimension $>1$.
\end{proof}
}

%
\subsection{Characteristic classes and Brauer classes of projective bundles}
We define next the invariant $\kappa(\PP)$,
for any holomorphic $\PP^{r-1}$-bundle, $r\geq 1$, over a complex variety $Y$, endowed with
the analytic topology. 
The definition is clear, if $\PP$ is
the projectivization of a vector bundle $E$, since $\kappa(E)$ is
independent of the choice of $E$. 
More generally, the Brauer class 
\[
\theta(\PP) \ \ \ \in \ \ \ H^2_{an}(Y,\StructureSheaf{Y}^*)
\] 
is the obstruction class to lifting $\PP$ to a holomorphic vector bundle.
The Brauer class $\theta(\PP)$ is the image of the class
$[\PP]\in H^1_{an}(Y,PGL_r)$, under the connecting homomorphism of
the short exact sequence of sheaves
\[
0\rightarrow \StructureSheaf{Y}^* \rightarrow GL_r(\StructureSheaf{}) 
\rightarrow PGL_r(\StructureSheaf{})
\rightarrow 0.
\]
Consider the dual bundle $\pi:\PP^*\rightarrow Y$.
The pullback $\pi^*\PP$ has a tautological hyperplane subbundle,
hence a divisor, hence a holomorphic line bundle
$\StructureSheaf{\pi^*\PP}(1)$.
The obstruction class $\theta(\PP)$ is in the kernel of 
$\pi^*:H^2_{an}(Y,\StructureSheaf{Y}^*)\rightarrow 
H^2_{an}(\PP^*,\StructureSheaf{\PP^*}^*)$ and 
the projective bundle $\pi^*\PP$ over $\PP^*$
is the projectivization of some vector bundle $\widetilde{E}$.
The class $\kappa(\widetilde{E})$ belongs to the
image of the injective homomorphism 
$\pi^*:H^*(Y,\RationalNumbers)\rightarrow H^*(\PP^*,\RationalNumbers)$, 
since $\kappa(\widetilde{E})$ restricts as $r$ to each fiber of $\pi$.
Define
\begin{equation}
\label{eq-kappa-PP}
\kappa(\PP) \ \ \ \in \ \ \ H^*(Y,\RationalNumbers)
\end{equation} 
as the unique class satisfying $\pi^*(\kappa(\PP))=\kappa(\widetilde{E})$. 

The class $\theta(\PP)$ is determined by a topological class,
which we now define.
Let $\mu_r$ be the group of $r$-th roots of unity. Denote
the corresponding local system by $\mu_r$ as well, and
let $\iota:\mu_r\rightarrow \StructureSheaf{}^*$ be the inclusion.
Let 
\begin{equation}
\label{eq-connecting-homomorphism-theta-tilde}
\tilde{\theta} \ : \ H^1_{an}(Y,PGL_r(\StructureSheaf{})) \ \ \ \rightarrow \ \ \
H^2(Y,\mu_r)
\end{equation}
be the connecting homomorphism of the short exact sequence
\[
0\rightarrow \mu_r \rightarrow SL_r(\StructureSheaf{})\rightarrow 
PGL_r(\StructureSheaf{}) \rightarrow 0. 
\]
Then the following equality clearly holds.
\begin{equation}
\label{eq-theta-factors-through-iota}
\theta(\PP) \ \ \ = \ \ \ \iota[\tilde{\theta}(\PP)].
\end{equation}

The exponential function $\exp:\ComplexNumbers\rightarrow \ComplexNumbers^*$ maps the subgroup $\frac{2\pi\sqrt{-1}}{r}\Integers$ of $\ComplexNumbers$ onto $\mu_r$ and induces the homomorphism
$\exp:H^2(Y,\frac{2\pi\sqrt{-1}}{r}\Integers)\rightarrow H^2(Y,\mu_r)$.
When $\PP$ is the projectivization of a vector bundle $V$ over $Y$, 
the following equality holds \cite[Lemma 2.5]{huybrechts-schroer}
\begin{equation}
\label{eq-tilde-theta-via-c-1}
\tilde{\theta}(\PP{V}) \ \ \ = \ \ \ 
\exp\left(\frac{-2\pi\sqrt{-1}}{r}c_1(V)\right).
\end{equation}

%
\subsection{Twisted sheaves}
\label{sec-twisted-sheaves}
\begin{defi}
\label{def-twisted-sheaves}
{\rm
Let $Y$ be a scheme or a complex analytic space,
$\U:=\{U_\alpha\}_{\alpha\in I}$ a covering, open in the complex 
or \'{e}tale topology,
and $\theta\in Z^2(\U,\StructureSheaf{Y}^*)$ a \v{C}ech $2$-cocycle.
A {\em $\theta$-twisted sheaf} consists of sheaves $E_\alpha$ of
$\StructureSheaf{U_\alpha}$-modules over $U_\alpha$,
for all $\alpha\in I$, and isomorphisms 
$g_{\alpha\beta}:(E_\beta\restricted{)}{U_{\alpha\beta}}
\rightarrow (E_\alpha\restricted{)}{U_{\alpha\beta}}$ 
satisfying the conditions:

(1) $g_{\alpha\alpha}=id$,

(2) $g_{\alpha\beta}=g_{\beta\alpha}^{-1}$, 

(3) $g_{\alpha\beta}g_{\beta\gamma}g_{\gamma\alpha}=
\theta_{\alpha\beta\gamma}\cdot id.$ 

\noindent
The $\theta$-twisted sheaf is {\em coherent}, if the $E_\alpha$ are.
}
\end{defi}

The abelian categories of $\theta$-twisted and $\theta'$-twisted
coherent sheaves are equivalent,
if the cocycles $\theta$ and $\theta'$ represent the same cohomology
class. The equivalence is not canonical, but the ambiguity is only
up to tensorization by an untwisted line bundle \cite{calduraru-thesis}. 

\begin{defi}
\label{def-equivalent-twisted-sheaves}
A $\theta$-twisted sheaf $E$ is said to be {\em equivalent} to a $\theta'$-twisted sheaf $F$ if 
the cocycles $\theta$ and $\theta'$ represent the same cohomology
class and there exists an equivalence of the two categories of coherent sheaves sending $E$ to $F$.
\end{defi}

The classes $\kappa_i$ of a twisted sheaf, defined below, 
are well defined for its equivalence class. 
We will often abuse terminology and refer
to a {\em $\theta$-twisted sheaf},  where $\theta$ is a class in
$H^2_{an}(Y,\StructureSheaf{Y}^*)$, meaning the equivalence class of
$\theta$-twisted sheaves, for different choices of 
\v{C}ech cocycles $\theta'$, representing the class $\theta$.

\begin{rem}
\label{rem-order-divides-rank}
Let $E$ be a $\theta$-twisted  
coherent torsion free sheaf  of rank $r$ over a complex manifold $Y$.
Observe that the determinant $\det(E)$ is a $\theta^r$-twisted line bundle.
Thus, $\theta^r$ is a coboundary. Consequently, 
the order of the class $[\theta]$,
of $\theta$ in $H^2_{an}(Y,\StructureSheaf{Y}^*)$,
divides the rank of every $\theta$-twisted torsion free sheaf $E$. 
\end{rem}

Assume $Y$ is a complex manifold. 
A projective $\PP^{r-1}$ bundle $\PP$ over $Y$ corresponds to a rank $r$
locally-free twisted coherent sheaf $E$, with twisting cocycle $\theta$ in
$Z^2(\U,\StructureSheaf{Y}^*)$, for some open covering $\U$ of $Y$.
The $\theta$-twisted sheaf $E$ is unique, up to tensorization by a 
line bundle. The characteristic class $\kappa(E):=\kappa(\PP)$ 
can be generalized for twisted sheaves, which are not locally free,
as we show next.

Let $\theta\in Z^2(\U,\StructureSheaf{Y}^*)$ be a two cocycle  and 
$E:=(E_\alpha,g_{\alpha\beta})$ a
$\theta$-twisted sheaf of rank $r>0$. 
We get a well defined\footnote{Details are provided in Section 2.2 of the first preprint version
arXiv:1105.3223v1 of this paper.} 
class $E^{\otimes r}\otimes \det(E)^{-1}$ in the $K$-group of coherent (untwisted) sheaves on $Y$, where the tensor product is taken in the $K$-ring taking into account the torsion sheaves in all degrees.
Let 
\[
Sqrt_r(x):= r +\frac{1}{r^r}(x-r^r) + \dots
\]
be the Taylor series of the branch of the $r$-th root function centered at $r^r$.
Set
\begin{equation}
\label{eq-kappa-in-terms-of-multi-tor-sheaves}
\kappa(E) \ \ \ := \ \ \ 
Sqrt_r\left(ch(
E^{\otimes r}\otimes \det(E)^{-1})
\right).
\end{equation}
If $E$ is untwisted, then the class $\kappa(E)$ is equal to $ch(E)\exp(-c_1(E)/r)$.
The above formula (\ref{eq-kappa-in-terms-of-multi-tor-sheaves}) is well defined for complexes of non-zero rank of $\theta$-twisted sheaves as well. The characteristic class $\kappa$ is again multiplicative,
$\kappa(E\otimes F)=\kappa(E)\kappa(F)$, where the tensor product is $K$-theoretic, and we pass to common refinements of the open covering in terms of which the co-cycles representing the Brauer classes are defined.

Note: The Chern character $ch(F)$ of a $\theta$-twisted sheaf, with 
a topologically trivial class $\theta$, was defined in 
\cite{huybrechts-stelari}, depending on a choice of a lift of $\theta$
to a class in $H^2(Y,\RationalNumbers)$. Another definition is provided in
\cite{lieblich-duke}.

\begin{lem}
\label{lemma-kappa-2-is-propotional-to-c-2-of-End}
Let $F$ be a reflexive coherent, possibly twisted, sheaf of rank $r$ over a complex manifold $X$. 
Then $c_2(\SheafEnd(F))=-2r\kappa_2(F).$
\end{lem}

\begin{proof}
The singular locus $Z$ of $F$ has complex codimension $\geq 3$ in $X$, since $F$ is reflexive, and so the restriction homomorphism
$\iota^*:H^4(X,\RationalNumbers)\rightarrow H^4(X\setminus Z,\RationalNumbers)$ is injective. Hence, it suffices to check the equality
$\iota^*c_2(\SheafEnd(F))=-2r\iota^*\kappa_2(F).$ We may thus assume that $F$ is locally free, possibly after replacing $X$ by $X\setminus Z$. Note that $\kappa_2(F^*)=\kappa_2(F)$, since $\kappa_i(F^*)=(-1)^i\kappa_i(F)$ for a locally free $F$. 
We have, 
\[
\kappa(\SheafEnd(F))=\kappa(F)\kappa(F^*)=(r+\kappa_2(F)+\dots)(r+\kappa_2(F^*)+\dots)=r^2+2r\kappa_2(F)+\dots
\] 
The claimed identity now follows from the equalities
$\kappa_2(\SheafEnd(F))=ch_2(\SheafEnd(F))=-c_2(\SheafEnd(F))$.
\end{proof}

While $\kappa$ is not additive, we do have the following statement that will be needed below.

\begin{lem}
\label{lemma-partial-additivity}
Let $Y$ be a compact K\"{a}hler manifold, $X$ a closed smooth complex submanifold, and 
$\delta:X\rightarrow Y$ the inclusion.  Let $E$ be a complex of non-zero rank of $\theta$-twisted coherent sheaves on $Y$ and
let $F$ be a complex of non-zero rank of  
$\delta^*\theta$-twisted sheaves on $X$. Assume that $c_1(F\otimes \delta^!E^\vee)=0$,
where the tensor product and the dualization are $K$-theoretic. Then 
\[
\kappa(E\oplus \delta_*F)=\kappa(E)+\delta_*(\kappa(F)td_\delta).
\]
\end{lem}

The special case $X=Y$ states that if  $c_1(E\otimes F^\vee)=0$ then $\kappa(E\oplus F)=\kappa(E)+\kappa(F)$.

\begin{proof}
$\kappa(E\oplus \delta_*F)\kappa(E^\vee)=
ch([E\oplus \delta_*F]\otimes E^\vee)=
ch(E\otimes E^\vee)+ch((\delta_*F)\otimes  E^\vee)=
\kappa(E)\kappa(E^\vee)+\delta_*(ch(F\otimes \delta^! E^\vee)td_\delta)=
\kappa(E)\kappa(E^\vee)+\delta_*(\kappa(F)\kappa(\delta^! E^\vee)td_\delta)=
[\kappa(E)+\delta_*(\kappa(F)td_\delta)]\kappa(E^\vee).$
The first equality is due to the assumed vanishing of $c_1(F\otimes \delta^!E^\vee)$, the third follows from Grothendieck-Riemann-Roch
and the sheaf $K$-theoretic projection formula $(\delta_!F)\otimes E^\vee\cong\delta_!(F\otimes\delta^!E^\vee)$, and the last is due to the cohomological projection formula.
The statement follows, since $\kappa(E^\vee)$ is invertible.
\end{proof}

%
\subsection{Sheaves of Azumaya algebras and their characteristic classes}
\label{sec-azumaya}

\begin{defi}
\label{def-Azumaya}
A {\em reflexive sheaf of Azumaya\footnote{Caution: 
The standard definition of a sheaf of
Azumaya $\StructureSheaf{X}$-algebras assumes that $E$ is a locally free
$\StructureSheaf{X}$-module, while we assume only that it is reflexive.} 
$\StructureSheaf{X}$-algebras of rank $r$} over a 
complex manifold $X$
is a sheaf $E$ of reflexive coherent $\StructureSheaf{X}$-modules, with a 
global section $1_E$, and an associative multiplication 
$m:E\otimes E\rightarrow E$ with identity $1_E$, admitting an open covering
$\{U_\alpha\}$ of $X$, and an isomorphism 
$\eta_\alpha:\restricted{E}{U_\alpha}\rightarrow \SheafEnd(F_\alpha)$
of unital associative algebras,
for some reflexive sheaf $F_\alpha$ of rank $r$, over each $U_\alpha$.
\end{defi}

From now on the term a sheaf of Azumaya algebras will mean
a reflexive sheaf of Azumaya $\StructureSheaf{X}$-algebras.
Fix a closed analytic subset $Z\subset X$, of codimension 
$\geq 3$, and set $U:=X\setminus Z$. 
A reflexive sheaf of Azumaya $\StructureSheaf{X}$-algebras 
is determined by its restriction to $U$. Hence, 
the set of isomorphism classes of reflexive
Azumaya $\StructureSheaf{X}$-algebras $E$ of rank $r$, 
which are locally free over $U$,
is in natural bijection with $H^1_{an}(U,PGL(r))$
\cite{milne}. Similarly, $H^1_{an}(U,PGL(r))$
parametrizes equivalence classes of coherent reflexive twisted
$\StructureSheaf{X}$-modules, which are locally free over $U$. 
We get a natural identification, of the set of 
isomorphism classes of reflexive sheaves of 
Azumaya $\StructureSheaf{X}$-algebras, with the set of equivalence classes of 
coherent reflexive twisted
$\StructureSheaf{X}$-modules. 


Let $E$ be a reflexive sheaf of Azumaya $\StructureSheaf{X}$-algebras, $m$ its multiplication, 
and $F$ a reflexive coherent twisted sheaf representing the equivalence class of $(E,m)$. 
Such a twisted sheaf $F$ exists, by \cite[Theorem 1.3.5]{calduraru-thesis}.
We set
\[
\kappa(E,m) \ \ \ := \ \ \ \kappa(F).
\]

\begin{caution}
Note that  $\kappa(E,m)$ is not equal to the class $\kappa(E)$ of the rank $r^2$
coherent sheaf $E$.
\end{caution}

%
\section{Monodromy invariant classes}
In Subsection \ref{sec-construction-of-the-classes-kappa-i-X} we construct the monodromy invariant classes 
$\kappa_i(\M)$ on a moduli space $\M$ of stable sheaves on a $K3$ surface $S$ in terms of the universal sheaf
over $S\times \M$.
In Subsection \ref{sec-Mon-invariant-clases-over-M-times-M}
we convolve the universal sheaf with its dual to obtain an object in the derived category of 
$\M\times \M$ and use it to construct the monodromy invariant classes over the product.
In Subsection \ref{sec-review-of-results-on-Monodromy} we prove the monodromy invariance of these classes.

%
\subsection{The rational Hodge classes $\kappa_i(X)$}
\label{sec-construction-of-the-classes-kappa-i-X}
Let $S$ be a projective $K3$ surface and $v\in K_{top}S$
a primitive class of positive rank  with $c_1(v)$ of type $(1,1)$.
Assume that $(v,v)\geq 2$.
There is a system of hyperplanes in the ample cone of $S$, called $v$-walls,
that is countable but locally finite \cite[Ch. 4C]{huybrechts-lehn-book}.
An ample class is called {\em $v$-generic}, if it does not
belong to any $v$-wall. Choose a $v$-generic ample class $H$. 
Then the moduli space $\M_H(v)$ is a
projective irreducible holomorphic symplectic manifold,
deformation equivalent to $S^{[n]}$, with $n=1+\frac{(v,v)}{2}$.
This result is due to several people, including 
Huybrechts, Mukai, O'Grady, and Yoshioka. It can be found in its final form in
\cite{yoshioka-abelian-surface}.

Let $f_1$ and $f_2$ be the projections on the
first and second factors of $S\times \M_H(v)$. 
Assume further that a universal sheaf $\E$ exists over $S\times \M_H(v)$. (This assumption will be dropped in Section \ref{sec-a-reflexive-sheaf-E-and-its-resolution}).

Let $e:K_{top}S\rightarrow K_{top}\M_H(v)$
be the homomorphism given by
\begin{equation}
\label{eq-e-x}
e_x \ \ \ := \ \ \ f_{2_!}\left(f_1^!(-x^\vee)\otimes [\E]\right).
\end{equation}
The class $e_x$ has rank $(v,x)$, in terms of the Mukai pairing
\begin{equation}
\label{eq-mukai-pairing-on-K-top}
(x,y) \ \ \ := \ \ \ -\chi(x^\vee\otimes y),
\end{equation}
for $x,y \in K_{top}S$. Let $v^\perp$ be the sublattice of $K_{top}S$ orthogonal to $v$.

Mukai defines a weight $2$ Hodge structure on 
$K_{top}S\otimes_\Integers\ComplexNumbers$ 
as follows. The $(2,0)$ summand is the  pull-back of $H^{2,0}(S)$, 
via the Chern character isomorphism $ch:K_{top}S\rightarrow H^*(S,\Integers)$, 
and the pullback of $H^0(S,\Integers)$ and $H^4(S,\Integers)$ are both of 
Hodge type $(1,1)$ \cite{mukai-hodge}. 
Recall that $H^2(\M_H(v),\Integers)$ is endowed with the Beauville-Bogomolov pairing.
The homomorphism 
\begin{eqnarray}
\label{eq-Mukai-isomorphism}
v^\perp & \rightarrow & H^2(\M_H(v),\Integers),
\\
\nonumber
x & \mapsto & c_1(e_x),
\end{eqnarray}
is an isometry and an isomorphism of
weight $2$ Hodge structures \cite{ogrady-hodge-str,yoshioka-abelian-surface}.

The {\em Mukai vector} of a class $v\in K_{top}S$ is the class 
$ch(v)\sqrt{td_S}\in H^*(S,\Integers)$. Following Mukai, we write the Mukai vector
of $v$ as a triple $(r,c_1(v),s)$, where the rank $r$ 
corresponds to the summand in $H^0(S,\Integers)$, while 
the summand in $H^4(S,\Integers)$ corresponds to
the integer $s$ times the class Poincare-dual to a point.
The Hirzebruch-Riemann-Roch Theorem yields
the equality
\[
(v,v) \ \ = \ \ c_1(v)^2-2rs.
\]

\begin{prop}
\label{prop-kappa-e-v-is-Mon-invariant}
The class $\kappa(e_v)$ is invariant under an index $2$
subgroup of $Mon(\M_H(v))$. 
If $i$ is even, then $\kappa_i(e_v)$ is $Mon(\M_H(v))$-invariant.
If $i$ is odd, then 
$\mbox{span}_\RationalNumbers\{\kappa_i(e_v)\}$ 
is $Mon(\M_H(v))$-invariant and $Mon(\M_H(v))$ acts on it via the character $\rho$ given in (\ref{eq-character-rho}). 
\end{prop}

The proposition is proven in Section
\ref{sec-review-of-results-on-Monodromy}
using results of \cite{markman-monodromy-I,markman-integral-constraints}. 
Proposition \ref{prop-kappa-e-v-is-Mon-invariant} yields a 
monodromy invariant pair of a class and its negative, denoted by  
\begin{equation}
\label{eq-a-Hodge-class-on-X}
\pm\kappa_i(X), 
\end{equation}
for any irreducible holomorphic symplectic manifold $X$
of $K3^{[n]}$-type, $n\geq 2$. 
The class $\kappa_i(X)$ is of type $(i,i)$,
by Lemma \ref{lemma-Mon-invariant-classes-are-of-Hodge-type}. 
Let $X^d$ be the $d$-th cartesian product of $X$. 


\begin{lem}
\label{lemma-Mon-invariant-classes-are-of-Hodge-type}
Let $\alpha\in H^{2i}(X^d,\ComplexNumbers)$ be a class, which is invariant under the diagonal
action of a finite index subgroup of $Mon(X)$. Then $\alpha$ is of Hodge type $(i,i)$.
\end{lem}

\begin{proof} The case $d=1$ of the statement is proven in 
\cite[Prop. 3.8 part 3]{markman-integral-constraints}.
We sketch the proof for the convenience of the reader.
We endow $Mon(X)$ with the Zariski topology induced by 
$GL(H^*(X,\ComplexNumbers))$. Let $\LieAlg{so}(H^2(X,\ComplexNumbers))$ be the Lie algebra associated to the Beauville-Bogomolov pairing.
The Lie algebra $\LieAlg{g}$ of the identity component of the Zariski closure of
$Mon(X)$ in $GL[H^*(X,\ComplexNumbers)]$ is equal to the image of
a faithful representation of the Lie algebra
$\LieAlg{so}(H^2(X,\ComplexNumbers))$ on $H^*(X,\ComplexNumbers)$,
constructed by Verbitsky  \cite[Theorem 7.1]{verbitsky} (see also \cite[Sec. 4]{looijenga-lunts}). 
The equality of these Lie algebras is proven in 
\cite[Lemma 4.11]{markman-monodromy-I}.
Verbitsky proved that 
the semi-simple endomorphism $h$ of $H^*(X,\ComplexNumbers)$, 
which acts on $H^{p,q}(X)$ by $\sqrt{-1}(p-q)$, 
is an element of the image of $\LieAlg{so}(H^2(X,\ComplexNumbers))$ \cite[Theorem 7.1]{verbitsky}, and is hence tangent to the
identity component of the Zariski closure of
$Mon(X)$. The latter component 
is also the identity component of the Zariski closure of
any finite index subgroup of $Mon(X)$, and in particular of the subgroup 
leaving the class $\alpha$ invariant. 
Hence, $\alpha$ belongs to the kernel of $\delta(h)$, where 
$\delta:\LieAlg{g}\rightarrow \LieAlg{gl}[H^*(X^d,\ComplexNumbers)]$
is the diagonal representation. Now
$\delta(h)=\sum_{i=1}^d id_{X^{i-1}}\otimes h \otimes id_{X^{d-i}}$,
which is the Hodge operator of $H^*(X^d,\ComplexNumbers)$. Hence,
$\alpha$ is of Hodge type $(i,i)$.
\end{proof}

%
\subsection{Monodromy invariant classes $\kappa_i(\F)$ 
over $\M(v)\times \M(v)$}
\label{sec-Mon-invariant-clases-over-M-times-M}
Set $\M:=\M_H(v)$.
Assume that a universal sheaf $\E$ exists over $S\times \M$.
A choice of a stable sheaf $G$ in $\M$ yields a
lift of the class $e_v$, given in (\ref{eq-e-x}),
to a class in the bounded derived category of coherent sheaves
$D^b_{Coh}(\M)$. Avoiding such a choice, we construct instead 
a natural class over $\M\times \M$. 
 
Let $\pi_{ij}$ be the projection from $\M\times S\times \M$ onto the
product of the $i$-th and $j$-th factors. 
Consider the following object in the bounded derived category of 
coherent sheaves over $\M\times \M$:
\begin{equation}
\label{eq-object-F}
\F \ \ \ := \ \ \ 
R\pi_{13_*}\left[
\pi_{12}^*\E^\vee\DerivedOtimes\pi_{23}^*\E
\right][1],
\end{equation} 
where the dual and the tensor product are taken in the derived category.
Let $\iota_G:\M\rightarrow \M\times \M$, be the embedding sending 
a point $[G']\in \M$ to $(G,G')$. Then
$\iota_G$ relates the class of $\F$ in $K_{top}(\M\times \M)$ to $e_v$:

\begin{lem}
\label{eq-e-v-is-restriction-of-F}
$e_v \ \ \ = \ \ \ \iota_G^![\F].$
\end{lem}

\begin{proof}
Denote by $\tilde{\iota}_G: S\times \M\hookrightarrow \M\times S\times \M$ 
the morphism given by $(x,G')\mapsto (G,x,G')$.
The Cohomology and Base Change Theorem yields the second equality 
below:\\
$
\iota_G^![-\F]=\iota_G^!\pi_{13_!}(\pi_{12}^!\E^\vee\otimes \pi_{23}^!\E)=
f_{2_!}\tilde{\iota}_G^!(\pi_{12}^!\E^\vee\otimes \pi_{23}^!\E)=
f_{2_!}(f_1^!G^\vee\otimes \E) = -e_v.
$
\end{proof}

\begin{prop}
\label{prop-kappa-F-is-mon-invariant}
The class $\kappa(\F)$ in $H^*(\M\times \M,\RationalNumbers)$
is invariant
under the diagonal action of a finite index subgroup of $Mon(\M)$. 
If $i$ is even, then $\kappa_i(\F)$ is $Mon(\M)$-invariant.
If $i$ is odd, then 
$\mbox{span}_\RationalNumbers\{\kappa_i(\F)\}$ 
is $Mon(\M)$-invariant and $Mon(\M)$ acts on it via the character $\rho$ given in (\ref{eq-character-rho}). 
\end{prop}

The proposition is proven in Section \ref{sec-review-of-results-on-Monodromy}
using results of \cite{markman-monodromy-I,markman-integral-constraints}.

\begin{lem}
\label{lem-c1-of-F}
$c_1(\F) \ \ \ = \ \ \ -\pi_1^*c_1(e_v)+\pi_2^*c_1(e_v)$.
\end{lem}

The lemma is proven in Section 
\ref{sec-review-of-results-on-Monodromy}.
When $v$ is the class of the ideal sheaf of a length $n$ subscheme, 
and $\E$ is the universal ideal sheaf, then 
$c_1(e_v)$ is half the class of the big diagonal in
$S^{[n]}$ \cite[Lemma 5.9]{markman-integral-constraints}. 

The object $\F$ fits in an exact triangle
\[
\SheafExt^1_{\pi_{13}}(\pi_{12}^*\E,\pi_{23}^*\E) 
\rightarrow \F \rightarrow 
\SheafExt^2_{\pi_{13}}(\pi_{12}^*\E,\pi_{23}^*\E)[-1]  \rightarrow
\SheafExt^1_{\pi_{13}}(\pi_{12}^*\E,\pi_{23}^*\E)[1].
\]
Furthermore, $\SheafExt^2_{\pi_{13}}(\pi_{12}^*\E,\pi_{23}^*\E)$ 
is isomorphic to the structure sheaf 
$\StructureSheaf{\Delta}$ of the diagonal
$\Delta\subset [\M\times \M]$, while 
$\SheafExt^1_{\pi_{13}}(\pi_{12}^*\E,\pi_{23}^*\E)$
is a reflexive sheaf of rank 
$(v,v)$, and is locally free away from $\Delta$
(Proposition \ref{prop-V}). 

\begin{prop}
\label{prop-kappa-E-is-mon-invariant}
The statement of Proposition \ref{prop-kappa-F-is-mon-invariant} holds if we substitute $E:=\SheafExt^1_{\pi_{13}}(\pi_{12}^*\E,\pi_{23}^*\E)$ for $\F$.
\end{prop}

The Proposition is proven in Section 
\ref{sec-review-of-results-on-Monodromy}.
\subsection{Proof of the monodromy invariance of 
$\kappa_i(X)$ and $\kappa_i(\F)$}
\label{sec-review-of-results-on-Monodromy}
We prove Propositions \ref{prop-kappa-e-v-is-Mon-invariant}
and \ref{prop-kappa-F-is-mon-invariant},  Lemma \ref{lem-c1-of-F}, and Proposition \ref{prop-kappa-E-is-mon-invariant}, 
after reviewing the necessary facts
about the monodromy group of $S^{[n]}$. 

Let $S$ be a $K3$ surface, $v\in K_{top}S$ a primitive class
with $c_1(v)$ of type $(1,1)$, and $H$ a $v$-generic line bundle.
Assume that $\M_H(v)$ is non-empty (in particular, $\rank(v)\geq 0$, 
$(v,v)\geq -2$, and $c_1(v)$ is effective if $\rank(v)= 0$). Then $\M_H(v)$ is a projective irreducible 
holomorphic symplectic manifold of
$K3^{[n]}$-type, with $2n=(v,v)+2$. Assume that $(v,v)\geq 2$. 
Then $H^2(\M_H(v),\Integers)$, endowed with the Beauville-Bogomolov pairing,
is Hodge isometric to $v^\perp\subset K_{top}S$, via Mukai's isometry
(\ref{eq-Mukai-isomorphism}).

We define next the orientation character of $O(K_{top}S)$.
A $4$-dimensional subspace $V$ of $K_{top}S\otimes_\Integers\RealNumbers$
is {\em positive definite}, if the Mukai pairing restricts to $V$ as a
positive-definite pairing. 
The positive cone $\C_+\subset K_{top}S\otimes_\Integers\RealNumbers$,
given by
\[
\C_+ \ \ \ := \ \ \ \{x \ : \ (x,x) > 0\},
\]
is homotopic to the unit $3$-sphere in any $4$-dimensional 
{\em positive definite} subspace \cite[Lemma 4.1]{markman-survey}. 
Hence $H^3(\C_+,\Integers)$ 
is isomorphic to $\Integers$ and is a natural character
\begin{equation}
\label{eq-cov}
cov \ : \ O(K_{top}S) \ \ \ \longrightarrow \ \ \ \{\pm1\}
\end{equation}
of the isometry group. 
Let $O^+(K_{top}S)$ be the kernel of $cov$.

Denote by $O(K_{top}S)_v$ the subgroup of isometries of $K_{top}S$ stabilizing $v$.  
Let $g$ be any isometry in $O(K_{top}S)_v$. 
It is {\em not} assumed to preserve the Hodge structure.
Denote by 
\[
(g\otimes 1) \ : \ K_{top}S\otimes K_{top}\M(v) \ \ \ \longrightarrow 
K_{top}S\otimes K_{top}\M(v)
\]
the homomorphism acting via the identity on the second factor. 
The K\"{u}nneth Theorem identifies $K_{top}S\otimes K_{top}\M(v)$
with $K_{top}[S\times \M(v)]$ \cite[Corollary 2.7.15]{atiyah-book}.
Assume that a universal sheaf $\E_v$ exists over $S\times \M(v)$
and let $[\E_v]$ be its class in $K_{top}[S\times \M(v)]$. 
Let $D:K_{top}S\rightarrow K_{top}S$ be the involution,
sending a class $x$ to its dual $x^\vee$. 
Set $n:=(v,v)/2+1$. We define a class in the middle cohomology
$H^{4n}(\M(v)\times \M(v),\Integers)$:
\[
\monclass(g) \ \ := \ \ 
\left\{
\begin{array}{ccc}
c_{2n}\left(
-\pi_{13_!}\left\{
\pi_{12}^!\left[
(g\otimes 1)[\E_{v_1}]
\right]^\vee\cup \pi_{23}^![\E_{v_2}]
\right\}
\right)
 & \mbox{if} & cov(g)=1,
\\
c_{2n}\left(
-\pi_{13_!}\left\{
\pi_{12}^!\left[
(Dg\otimes 1)[\E_{v_1}]
\right]\cup \pi_{23}^![\E_{v_2}]
\right\}
\right)
 & \mbox{if} & cov(g)=-1.
\end{array}
\right.
\]
Denote by 
\begin{equation}
\label{eq-monrep-g}
\monrep(g) \ : \ H^*(\M(v),\Integers) \ \ \ 
\longrightarrow \ \ \ H^*(\M(v),\Integers)
\end{equation}
the homomorphism obtained from $\monclass(g)$
using the K\"{u}nneth and Poincare-Duality Theorems. 

\begin{thm}
\label{thm-symmetries-of-moduli-spaces}
\cite[Theorems 1.2 and 1.6]{markman-monodromy-I}
\begin{enumerate}
\item
\label{thm-item-mon-g-is-a-monodromy-operator}
The endomorphism $\monrep(g)$ is an algebra automorphism and
a monodromy operator. 
\item
The assignment 
\begin{equation}
\label{eq-functor-mon}
\monrep \ : \ O(K_{top}S)_v \ \ \ \longrightarrow \ \ \ Mon(\M(v)),
\end{equation}
sending an isometry $g$ to the operator $\monrep(g)$, 
is a group homomorphism. The homomorphism
is injective, if $(v,v)\geq 4$, and its kernel
is generated by the involution 
\begin{equation}
\label{eq-reflection-by-v}
w \ \mapsto \ -w+(w,v)v, 
\end{equation}
if $(v,v)=2.$
\item
\label{thm-item-mon-g-sends-a-universal-classes-to-such}
There exists a topological complex line bundle $\ell_g$ on
$\M_{H}(v)$ satisfying one of the following equations: 
\[
\begin{array}{ccccc}
(g\otimes \monrep(g))[\E_{v}] & = & 
[\E_{v}]\otimes f_2^*\ell_g,
 & \mbox{if} & cov(g)=1,
\\
((D\circ g)\otimes \monrep(g))[\E_{v}] & = & 
[\E_{v}]^\vee\otimes f_2^*\ell_g,
 & \mbox{if} & cov(g)=-1.
\end{array}
\]
\end{enumerate}
\end{thm}

The action of $\monrep(g)$ on $K_{top}(\M(v))$, in part
\ref{thm-item-mon-g-sends-a-universal-classes-to-such} of the Theorem,
is constructed as follows.
The Chern character homomorphism  
$ch  :  K_{top}\M(v)  \rightarrow H^*(\M(v),\RationalNumbers)$
is injective,
since $K_{top}\M(v)$ is torsion free \cite{markman-integral-generators}.
The homomorphism $ch$ is monodromy equivariant; hence it 
maps $K_{top}\M(v)$ to a 
$\monrep(g)$-invariant subalgebra, for all $g\in O(K_{top}S)_v$, by part 
\ref{thm-item-mon-g-is-a-monodromy-operator} 
of Theorem \ref{thm-symmetries-of-moduli-spaces}.
Denote by $\monrep_g$ the corresponding monodromy automorphism of
$K_{top}\M(v)$. 
Part \ref{thm-item-mon-g-sends-a-universal-classes-to-such}
of the Theorem can be rephrased in terms of the homomorphism
$e:K_{top}S\rightarrow K_{top}\M(v)$, given in (\ref{eq-e-x}):
\begin{equation}
\label{eq-mon-equivariance-of-e}
\monrep_g(e_{g^{-1}(x)}) \ \ \ = \ \ \ 
\left\{
\begin{array}{ccc}
e_x\otimes \ell_g, 
& \mbox{if} & cov(g)=1,
\\
(e_{x})^\vee\otimes \ell_g, & \mbox{if} & cov(g)=-1.
\end{array}
\right.
\end{equation}
Consequently, the line bundle $\ell_g$ is determined by the following 
formula:
\begin{equation}
\label{eq-c-1-ell-g}
c_1(\ell_g) \ \ \ = \ \ \
\frac{\monrep_g(c_1(e_v))- cov(g)\cdot c_1(e_v)}{(v,v)}.
\end{equation}

Let $Mon^2(\M(v))$  
be the image in $O[H^2(\M(v),\Integers)]$ of $Mon(\M(v))$
under the restriction homorphism from $H^*(\M(v),\Integers)$ to 
$H^2(\M(v),\Integers)$. 

\begin{thm} 
\label{thm-monodromy-operator-is-determined-by-its-weight-2-action}
The restriction homomorphism
$Mon(\M(v))\rightarrow Mon^2(\M(v))$ is an isomorphism.
\end{thm}

\begin{proof}
The statement was proved in \cite[Prop. 1.9]{markman-integral-constraints} conditional on 
the Global Torelli Theorem. The latter was later proved by Verbitsky  \cite{huybrechts-bourbaki,verbitsky-torelli}.
\end{proof}

\begin{thm}
\label{thm-weight-2-monodromy}
The homomorphism 
$\monrep:O(K_{top}S)_v\rightarrow Mon(\M(v))$ is surjective.
It is an isomorphism, if $(v,v)\geq 4$, and its kernel is
generated by the involution (\ref{eq-reflection-by-v}),
if $(v,v)=2$.
\end{thm}

\begin{proof}
Let 
$
\monrep^2:O(K_{top}S)_v\rightarrow Mon^2(\M(v))
$
be the composition of $\monrep$ with the restriction homomorphism.
If we replace $\monrep$ by $\monrep^2$ in the statement of the theorem we obtain a statement that was proved in 
\cite[Theorem 1.2 and Lemma 4.2]{markman-integral-constraints}. 
The theorem now follows from Theorem \ref{thm-monodromy-operator-is-determined-by-its-weight-2-action}.
\end{proof}

\hide{
Theorems \ref{thm-weight-2-monodromy} and
\ref{thm-monodromy-operator-is-determined-by-its-weight-2-action}
imply that the homomorphism $\monrep$,
given in (\ref{eq-functor-mon}), is surjective, if $2\leq n\leq 3$, since 
$H^*(S^{[2]},\RationalNumbers)$ is generated by $H^2$, and
$H^*(S^{[3]},\RationalNumbers)$  is generated by $H^2$ and $H^4$
\cite[Lemma 10]{markman-diagonal}.
For $n\geq 4$, we have: 

\begin{thm}
\cite[Corollary 1.6]{markman-integral-constraints}
assuming the results of \cite{lehn-sorger-generators}).
\begin{enumerate}
\item
The image of the homomorphism $\monrep:O(K_{top}S)_v\rightarrow Mon(\M(v))$,
given in (\ref{eq-functor-mon}), 
is a subgroup of $Mon(\M(v))$ of index $\leq 2$. 
\item
The homomorphism $\monrep$ is surjective, 
if $n:=\dim_{\ComplexNumbers}(\M(v))/2$ 
is congruent to $0$ or $1$ modulo $4$.
\end{enumerate}
\end{thm}
}

{\em Proof of Proposition \ref{prop-kappa-e-v-is-Mon-invariant}:}
If $n\geq 3$, the isomorphism $\monrep:O(K_{top}S)_v\rightarrow Mon(\M(v))$ of Theorem \ref{thm-weight-2-monodromy}
conjugates the character $cov$, given in (\ref{eq-cov}),  to the character $\rho$ given in (\ref{eq-character-rho}), by \cite[Lemma 4.2]{markman-integral-constraints}. When $n=2$, $\rho$ is the trivial character, since the discriminant group has order $2$ and so multimplication by $-1$ is the identity. Furthermore, in that case the homomorphism $\monrep$ maps the kernel of the character $cov$ isomorphically onto $Mon(\M(v))$.

It remains to prove that the pair 
$\{\kappa(e_v),\kappa\left((e_v)^\vee\right)\}$ 
is invariant under the image of
$O(K_{top}S)_v$ in $Mon(\M(v))$ via $\monrep$  and is permuted according to the character $cov$, as $\monrep$ is surjective, by Theorem \ref{thm-weight-2-monodromy}.
The $O(K_{top}S)_v$-invariance of the pair $\{\kappa(e_v),\kappa\left((e_v)^\vee\right)\}$  follows from the reformulation (\ref{eq-mon-equivariance-of-e})
of part \ref{thm-item-mon-g-sends-a-universal-classes-to-such} of
Theorem \ref{thm-symmetries-of-moduli-spaces}, and the fact that $g(v)=v$.
\EndProof

\medskip
{\em Proof of Proposition \ref{prop-kappa-F-is-mon-invariant}:}
It suffices to show that the pair 
$\{\kappa(\F),\kappa\left(\F^\vee\right)\}$ 
is invariant under the image of
$O(K_{top}S)_v$ in $Mon(\M(v))$ via $\monrep$ and is permuted according to the character $cov$, by the surjectivity of $\monrep$ in Theorem \ref{thm-weight-2-monodromy} and the relation between $cov$ and $\rho$ explained in the proof of 
Proposition \ref{prop-kappa-e-v-is-Mon-invariant}.
Denote by 
\[
D_\M \ : \ K_{top}\M(v) \ \ \ \rightarrow \ \ \ K_{top}\M(v)
\]
the duality involution $y\mapsto y^\vee$ and by $D_S$ 
the duality involution of $K_{top}S$. 
Note the equality 
\begin{equation}
\label{eq-duality-operator-factors-as-tensor-product}
[\E_v]^\vee=(D_S\otimes D_\M)[\E_v].
\end{equation} 
Caution: while $D_\M$ commutes with $Mon(\M(v))$, 
$D_S$ does {\em not} commute with $O(K_{top}S)_v$. 
The class $[\F]$ 
is the image in $K_{top}\M(v)\otimes K_{top}\M(v)$ of the class 
$\left\{(1\otimes D_\M)[\E_v]\right\}\otimes [\E_v]$ via the contraction 
with the Mukai pairing:
\begin{eqnarray*}
\nonumber
\psi : 
[K_{top}S \otimes K_{top}\M(v)]\otimes [K_{top}S \otimes K_{top}\M(v)]
& \rightarrow & 
K_{top}\M(v)\otimes K_{top}\M(v)
\\
x_1\otimes y_1 \otimes x_2 \otimes y_2 & \mapsto & 
-\chi(x_1^\vee\otimes x_2)y_1\otimes y_2
\end{eqnarray*} 
The equality 
$
\psi=\psi\circ (g\otimes 1\otimes g\otimes1)
$ holds, for any isometry $g$ of the Mukai lattice.
Hence, the following equality holds:
\[
\psi\left\{(g\otimes \monrep_g\circ D_\M)[\E_v]\otimes 
(g\otimes \monrep_g)[\E_v]\right\} \ =
\psi\left\{(1\otimes \monrep_g\circ D_\M)[\E_v]\otimes 
(1\otimes \monrep_g)[\E_v]\right\}.
\]
The right hand side is $(\monrep_g\otimes\monrep_g)[\F]$,
while the left hand side is equal to 
\[
\left\{
\begin{array}{ccc}
\psi\left\{
(1\otimes D_\M)([\E_v]\otimes f_2^!\ell_g)\otimes [\E_v]\otimes f_2^!\ell_g
\right\}, & \mbox{if} & cov(g)=1,
\\
\psi\left\{
([\E_v]\otimes f_2^!\ell_g^\vee)\otimes 
(1\otimes D_\M)([\E_v]\otimes f_2^!\ell_g^\vee)
\right\}, & \mbox{if} & cov(g)=-1,
\end{array}
\right.
\]
by part \ref{thm-item-mon-g-sends-a-universal-classes-to-such} of
Theorem \ref{thm-symmetries-of-moduli-spaces} (use also Equation (\ref{eq-duality-operator-factors-as-tensor-product}) when $cov(g)=-1$). 
The latter contractions simplify to
\begin{equation}
\label{eq-mon-equivariance-of-F}
(\monrep_g\otimes\monrep_g)[\F]  = 
\left\{
\begin{array}{ccc}
[\F]\otimes \pi_1^!(\ell_g^\vee)\otimes \pi_2^!\ell_g, 
& \mbox{if} & cov(g)=1,
\\
\left([\F]\right)^\vee\otimes \pi_1^!(\ell_g^\vee)\otimes \pi_2^!\ell_g, 
& \mbox{if} & cov(g)=-1,
\end{array}
\right.
\end{equation}
by the projection formula. We conclude that the pair 
$\{\kappa(\F),\kappa(\F^\vee)\}$ is $\monrep(g)$-invariant.
\EndProof

\medskip
{\em Proof of Lemma \ref{lem-c1-of-F}:}
For every $g\in O^+K_{top}S)_v$, we have:
\begin{eqnarray*}
& & 
(\monrep_g\otimes\monrep_g)(c_1(\F))  
\stackrel{(\ref{eq-mon-equivariance-of-F})}{=} 
c_1(\F) - (v,v)[\pi_1^*c_1(\ell_g) - \pi_2^*c_1(\ell_g)]
\stackrel{(\ref{eq-c-1-ell-g})}{=} 
\\
& &
c_1(\F) - \pi_1^*[\monrep_g(c_1(e_v))-c_1(e_v)] + 
\pi_2^*[\monrep_g(c_1(e_v))-c_1(e_v)].
\end{eqnarray*}
Consequently, $c_1(\F)+\pi_1^*(c_1(e_v))-\pi_2^*c_1(e_v)$
is $O^+(K_{top}S)_v$-invariant. 
The $O^+(K_{top}S)_v$-invariant subspace of $H^2(\M(v)\times \M(v))$
vanishes, since the latter is the direct sum of
two copies of the non-trivial irreducible $O^+(K_{top}S)_v$-module
$H^2(\M(v))$.
\EndProof

\medskip
{\em Proof of Proposition \ref{prop-kappa-E-is-mon-invariant}:}
Let $\delta:\M\rightarrow \M\times \M$ be the diagonal embedding. 
The sheaf $E$ is the sheaf cohomology in degree zero of $\F$ and the sheaf in degree $1$ is 
$\SheafExt^2_{\pi_{13}}(\pi_{12}^*\E,\pi_{23}^*\E)$, which is isomorphic to $\delta_*\StructureSheaf{\M}$.
We have
\[
\kappa(\F):=[ch(E)-ch(\delta_*\StructureSheaf{\M})]\exp\left(\frac{-c_1(\F)}{2n-2}\right)=
\kappa(E)-\delta_*\left[\delta^*\exp\left(\frac{-c_1(\F)}{2n-2}\right)td_\delta\right]=\kappa(E)-\delta_*(td_\delta),
\]
where the second euality follows from Grothendieck-Riemann-Roch and the projection formula and the third from the vanishing of
$\delta^*c_1(\F)$ proven in Lemma \ref{lem-c1-of-F}. 
The difference $\kappa(\F)-\kappa(E)$ is thus
$-\delta_*(td_\delta)$, 
and 
is invariant under the diagonal monodromy action. We claim that the graded summand of $\delta_*(td_\delta)$ in $H^{2i}(\M\times\M)$
vanish for odd $i$.
The claim would follow once we show that the  graded summands of the Todd class of $\delta$ in degree $2i$ vanish, for odd $i$, since the degree of $\delta_*$ is divisible by $4$.
Indeed, the odd Chern classes of the normal bundle of the diagonal vanish, the normal bundle being the tangent bundle hence self dual, and so the odd graded summands of the Todd class of $\delta$ vanish.
Consequently, the monodromy invariance properties of $\kappa_i(E)$ follow from those of $\kappa_i(\F)$. 
\EndProof

%
\section{A reflexive sheaf  over $\M(v)\times \M(v)$}
\label{sec-a-reflexive-sheaf-E-and-its-resolution}

Keep the notation of Section \ref{sec-Mon-invariant-clases-over-M-times-M}.
Let $\F$ be the object over $\M\times \M$ constructed in Equation (\ref{eq-object-F}). We consider in this section the case, where the universal sheaf $\E$ over $S\times \M$ is twisted by the pullback a Brauer class $\theta$ on $\M$, 
so that the object $\F$ is $\pi_1^*(\theta^{-1})\pi_2^*\theta$-twisted.
We show that the first sheaf cohomology of $\F$ is a reflexive $\pi_1^*(\theta^{-1})\pi_2^*\theta$-twisted sheaf $E$ over $\M\times \M$,
singular along the diagonal.
We then resolve the singularities of $E$ via a 
locally free twisted sheaf $V$, over the blow-up of the
diagonal in $\M\times \M$. 

Let $\beta:B\rightarrow [\M\times \M]$ be the blow-up of $\M\times \M$ 
along the diagonal $\Delta$, $D:=\PP(T\Delta)$ the exceptional divisor, 
$\iota:D\hookrightarrow B$ the closed immersion, 
$\delta:\Delta\hookrightarrow \M\times \M$ the diagonal embedding, 
$p:D\rightarrow \Delta$ the bundle map, 
$\ell$  the tautological line subbundle of $p^*T\Delta$, and 
$\ell^\perp$ the symplectic-orthogonal subbundle of $p^*T\Delta$.
Let $\tau$ be the involution of $\M\times \M$, 
interchanging the two factors, 
and $\tilde{\tau}$ the induced involution of $B$. 
Note that $\tau^*(\F)=\F^\vee$, by Grothendieck-Serre's Duality, 
and the triviality of the relative canonical line bundle $\omega_{\pi_{13}}$. 
Note that the object $L\delta^*\F$ and the sheaf $\delta^*E$ are untwisted as the  the Brauer class 
$\pi_1^*(\theta^{-1})\pi_2^*\theta$ restricts to the diagonal as the trivial class.
Set $E^*:=\SheafHom(E,\StructureSheaf{\M\times \M})$ and
$E^\vee:=R\SheafHom(E,\StructureSheaf{\M\times \M})$.

\begin{prop}
\label{prop-V}
\begin{enumerate}
\item
\label{lemma-item-SheafExt-1-is-reflexive}
The twisted sheaf $E:=\SheafExt^1_{\pi_{13}}(\pi_{12}^*\E,\pi_{23}^*\E)$ is reflexive of rank $(v,v)$. Furthermore, 
$\kappa(E^*)=\kappa(E^\vee)$.
\item
\label{lemma-item-tors}
$E$ restricts to $[\M\times \M]\setminus \Delta$ as a locally free sheaf.
We have the following isomorphism:
\begin{equation}
\label{eq-restriction-of-E-to-diagonal}
\delta^*E  \ \ \ \cong \ \ \  \left(\Wedge{2}T^*\M\right)/\StructureSheaf{\M}\cdot\sigma, 
\end{equation}
where $\sigma$ is the symplectic form. For $i>0$, we have
\begin{equation}
\label{eq-Tor-i-is-wedge-i+2}
\Tor_i^{\M\times \M}(E,\delta_*\StructureSheaf{\M}) \ \ \ \cong \ \ \ \delta_*\Wedge{i+2}T^*\M.
\end{equation}
\item
\label{lemma-item-V-is-locally-free}
The quotient 
\begin{equation}
\label{eq-vector-bundle-over-B}
V \ \ \ := \ \ \ 
\left[\beta^*E\right](D)/tor,
\end{equation}
by the torsion subsheaf, 
is a locally free sheaf of rank $(v,v)$ over $B$.
\item
\label{lemma-item-direct-images-of-V}
$\beta_*(V)\cong E$
and 
$R^i\beta_{*}(V)=0$, for $i>0$. 
\item
\label{lemma-part-pull-back-of-V-is-its-dual}
$\tilde{\tau}^*V$ is isomorphic to $V^*$. 
\item
\label{prop-item-restriction-of-V-is-a-symplectic-vb}
The restriction $\restricted{V}{D}$ is naturally 
identified with the sub-quotient 
\begin{equation}
\label{eq-tautological-sub-quotient}
[\ell^\perp/\ell].
\end{equation}
In particular, $\restricted{V}{D}$ 
is a symplectic vector bundle.
\end{enumerate}
\end{prop}



%
The following lemmas will be used in the proof of Proposition \ref{prop-V}.

\begin{lem}
\label{lem-facts-about-tautological-subquotient}
The following natural homomorphism is surjective: 
\begin{equation}
\label{eq-a-surjective-evaluation-homomorphism}
p^*p_*\left([\ell^\perp/\ell]\otimes \ell^*\right)\rightarrow 
[\ell^\perp/\ell]\otimes \ell^*.
\end{equation}
\end{lem}

\begin{proof}
We identify each of the vector bundles 
$T\Delta$ and $[\ell^\perp/\ell]$ with its dual, via the symplectic 
forms. We have the short exact sequence
\[
0\rightarrow [\ell^\perp/\ell]\otimes \ell^* \rightarrow 
[p^*T^*\Delta/\ell]\otimes \ell^*\rightarrow \ell^{-2}\rightarrow 0.
\]
$p_*\left([\ell^\perp/\ell]\otimes \ell^*\right)\cong 
\ker\left[p_*(\{p^*T^*\Delta\otimes \ell^*\}/\StructureSheaf{})
\rightarrow p_*(\ell^{-2})\right]$, which is naturally isomorphic to
the quotient 
$[\Wedge{2}{T^*\Delta}]/\StructureSheaf{}$, by the line-sub-bundle 
spanned by the symplectic form. 
The homomorphism (\ref{eq-a-surjective-evaluation-homomorphism})
is dual to the wedge product \ \ \ 
$
[\ell^\perp/\ell]\otimes \ell \ \ \ \rightarrow \ \ \ 
p^*([\Wedge{2}{T\Delta}]/\StructureSheaf{}),
$
which is clearly injective.
\end{proof}

\begin{lem}
\label{lemma-hom-from-ell-perp-to-ell-perp-dual}
$H^0((\ell^\perp\otimes\ell^\perp)^*)$ is one dimensional.
\end{lem}

\begin{proof}
It suffices to prove that $p_*((\ell^\perp\otimes\ell^\perp)^*)$ is isomorphic to $T\Delta\otimes T\Delta$, as 
$H^0(T\Delta\otimes T\Delta)\cong\End(T\Delta)$ is one dimensional, by the stability of $T\Delta$. 
We have the short exact sequences
\[
0\rightarrow (\ell^\perp)^*\otimes \ell\rightarrow (\ell^\perp)^*\otimes p^*T\Delta \rightarrow (\ell^\perp\otimes\ell^\perp)^*
\rightarrow 0,
\]
and
\[
0\rightarrow \ell^2\rightarrow p^*T\Delta\otimes \ell\rightarrow (\ell^\perp)^*\otimes \ell\rightarrow 0.
\]
$R^ip_*\ell^j$ vanishes for all $i$ and all $1\leq j\leq 2n-1$, by Kodaira's Vanishing Theorem.
Hence, $R^ip_*[(\ell^\perp)^*\otimes \ell]$ vanishes for all $i$ and
$p_*((\ell^\perp\otimes\ell^\perp)^*)$ is isomorphic to $p_*[(\ell^\perp)^*\otimes p^*T\Delta]$, and hence to
$p_*((\ell^\perp)^*)\otimes T\Delta$. The latter is isomorphic to $T\Delta\otimes T\Delta$ by applying the functor $p_*$ to 
the exact sequence $0\rightarrow \ell\rightarrow p^*T\Delta\rightarrow(\ell^\perp)^*\rightarrow 0$
and the vanishing of $R^ip_*\ell$, for all $i$.
\end{proof}

The proof of Proposition \ref{prop-V}
requires a review of the following construction carried out in 
\cite{markman-diagonal}. 
There exists a (non-canonical) complex
\begin{equation}
\label{eq-locally-free-complex-representing-F}
V_{-1}\LongRightArrowOf{g} V_0 \LongRightArrowOf{f} V_1,
\end{equation}
of locally free $\pi_1^*(\theta^{-1})\pi_2^*\theta$-twisted sheaves over $\M\times \M$, representing
the object $\F$ \cite{lange}. 
The sheaf homomorphism $g$ is injective, since 
$\SheafExt^0_{\pi_{13}}(\pi_{12}^*\E,\pi_{23}^*\E)$ vanishes.
The middle cohomology sheaf $\ker(f)/Im(g)$ is isomorphic to
$\SheafExt^1_{\pi_{13}}(\pi_{12}^*\E,\pi_{23}^*\E)$, and 
$\mbox{coker}(f)$ is isomorphic to 
$\SheafExt^2_{\pi_{13}}(\pi_{12}^*\E,\pi_{23}^*\E)$,
and hence also to $\delta_*\StructureSheaf{\Delta}$.
Furthermore, the dual complex represents the pullback $\tau^*(\F)$ of
the object $\F$. 
In particular, $\mbox{coker}(g^*)$ is also
isomorphic to $\delta_*\StructureSheaf{\Delta}$.

\begin{claim}
\begin{eqnarray}
\label{eq-Ext-1-Im-f-vanishes}
\SheafExt^1(\Im(f),\StructureSheaf{\M\times\M}) & = & 0,
\\
\label{eq-ker-f-is-reflexive}
\ker(f)^* & \cong & \coker(f^*),
\\
\label{eq-ker-g-*-is-reflexive}
\ker(g^*)^* & \cong & \coker(g).
\end{eqnarray}
\end{claim}

\begin{proof}
Consider the long exact sequence of extension sheaves, obtained by applying
$\SheafHom(\bullet,\StructureSheaf{\M\times\M})$ to the short exact sequence
\[
0 \rightarrow \Im(f) \rightarrow V_1 \rightarrow \StructureSheaf{\Delta}
\rightarrow 0.
\]
$\SheafExt^i(V_1,\StructureSheaf{\M\times\M})=0$, for $i>0$, and 
$\SheafExt^i(\StructureSheaf{\Delta},\StructureSheaf{\M\times\M})=0$,
for $0\leq i <\dim(\M)=2n$, by the Local Duality Theorem. 
The vanishing (\ref{eq-Ext-1-Im-f-vanishes}) follows. 

Applying $\SheafHom(\bullet,\StructureSheaf{\M\times\M})$ to the short 
exact sequence
\[
0 \rightarrow \ker(f) \rightarrow V_0 \rightarrow  \Im(f) \rightarrow 0,
\]
we get the short exact sequence
\[
0\rightarrow V_1^* \LongRightArrowOf{f^*} V_0^* \longrightarrow \ker(f)^*
\rightarrow 0,
\]
by the vanishing (\ref{eq-Ext-1-Im-f-vanishes}).
Equation (\ref{eq-ker-f-is-reflexive}) follows.

Equation (\ref{eq-ker-g-*-is-reflexive}) is the analogue of 
Equation (\ref{eq-ker-f-is-reflexive}) for the dual of the complex
(\ref{eq-locally-free-complex-representing-F}).
\end{proof}

The pullback of the complex (\ref{eq-locally-free-complex-representing-F}) via the diagonal embedding $\delta$ is
equivalent to the object $R_{f_2,*}(\E^\vee\otimes \E)[1]$ in $D^b(\M)$. In particular,
$\ker(\delta^*g)\cong \SheafExt_{f_2}^0(\E,\E)\cong\StructureSheaf{\M}$ and $\coker(\delta^*f)\cong \SheafExt_{f_2}^2(\E,\E)$ is its dual, by Grothendieck-Verdier Duality, hence is isomorphic to $\StructureSheaf{\M}$ as well.
Let $K$ be the kernel of 
$\restricted{g}{\Delta}:(V_{-1}\restricted{)}{\Delta}\rightarrow 
(V_{0}\restricted{)}{\Delta}$ and $F$ the image of 
$\restricted{f}{\Delta}:(V_{0}\restricted{)}{\Delta}\rightarrow 
(V_{1}\restricted{)}{\Delta}$.
Then $K$ and $(V_{1}\restricted{)}{\Delta}/F$ are both isomorphic to
$\StructureSheaf{\Delta}$. 
Let $U_{-1}$ be the subsheaf of $(\beta^*V_{-1})(D)$, whose sections
restrict to $D$ as sections of $[\iota_*(p^*K)](D)$. We get
the short exact sequence:
\begin{equation}
\label{eq-sequence-defining-U-minus-1}
0\rightarrow \beta^*V_{-1} \rightarrow U_{-1} \rightarrow 
[\iota_*(p^*K)](D) \rightarrow 0.
\end{equation}
Define $U_1\subset \beta^*V_1$ as the subsheaf, whose sections restrict
to $D$ as sections of $\iota_*(p^*F)$. It fits in 
the short exact sequence:
\begin{equation}
\label{eq-sequence-defining-U-1}
0\rightarrow U_1 \rightarrow \beta^*V_1\rightarrow 
\iota_*(p^*\mbox{coker}(f))\rightarrow 0.
\end{equation}
The section $\beta^*g$ of $\Hom(U_{-1}(-D),\beta^*V_0)$ vanishes along the divisor $D$ and hence defines 
a section $\tilde{g}$ of $\Hom(U_{-1},\beta^*V_0)$.
We get the complex of vector bundles over $B$
\[
U_{-1}\LongRightArrowOf{\tilde{g}} \beta^*V_0 \LongRightArrowOf{\tilde{f}}
U_1,
\]
where  $\tilde{f}$ is surjective. The dual of the above complex is obtained from the dual of the complex 
(\ref{eq-locally-free-complex-representing-F}) via the analogous construction. 
Hence, $\tilde{g}^*$ is surjective as well. 
Both $U_{-1}$ and $U_1$ are locally free $\StructureSheaf{B}$-modules. 
Set 
\begin{equation}
\label{eq-V-as-a-subquotient}
\widetilde{V} \ \ := \ \ \ker(\tilde{f})/Im(\tilde{g}).
\end{equation}
Then $\widetilde{V}$ is locally free as well. 
We will see in the course of the proof of Proposition \ref{prop-V} that $\widetilde{V}$ is isomorphic to the sheaf $V$ given in Equation (\ref{eq-vector-bundle-over-B}). 

\begin{claim}
\label{claim-higher-direct-images-of-U-s}
\begin{enumerate}
\item
\label{claim-item-U-minus-1}
$\beta_*(U_{-1})\cong V_{-1}$, and $R^i\beta_{*}(U_{-1})=0$, for $i>0$.
\item
\label{claim-item-U-1}
$\beta_*(U_{1})\cong Im(f)$, and $R^i\beta_{*}(U_{1})=0$, for $i>0$.
\item
\label{claim-item-ker-tilde-f}
$\beta_*(\ker(\tilde{f}))\cong\ker(f)$, and 
$R^i\beta_{*}(\ker(\tilde{f}))=0$, for $i>0$.
\end{enumerate}
\end{claim}

\begin{proof}
\ref{claim-item-U-minus-1})
The higher direct images $R^ip_{*}(\StructureSheaf{D}(D))$ vanish,
for $i\geq 0$. This vanishing implies Part \ref{claim-item-U-minus-1}, 
using the long exact sequence of higher direct images via $\beta$,
associated to the short exact sequence 
(\ref{eq-sequence-defining-U-minus-1}). 

\ref{claim-item-U-1}) 
The push-forward $p_*\StructureSheaf{D}$ is isomorphic to 
$\StructureSheaf{\Delta}$, and all the higher direct images vanish. 
Part \ref{claim-item-U-1} follows from the long exact sequence of 
higher direct images via $\beta$, associated to the short exact sequence 
(\ref{eq-sequence-defining-U-1}).

Part \ref{claim-item-ker-tilde-f} follows from part \ref{claim-item-U-1}
using the long exact sequence of 
higher direct images via $\beta$, associated to the short exact sequence 
$
0\rightarrow \ker(\tilde{f})\rightarrow \beta^*V_0\RightArrowOf{\tilde{f}} 
U_1 \rightarrow 0.
$
\end{proof}

{\em Proof of Proposition \ref{prop-V}:}\\
Part \ref{lemma-item-SheafExt-1-is-reflexive}) The sheaf $\ker(g^*)$ is reflexive, being a saturated subsheaf of a locally free sheaf.
Applying $\SheafHom(\bullet,\StructureSheaf{\M\times\M})$ to the
short exact sequence
\[
0\rightarrow \SheafExt^1_{\pi_{13}}(\pi_{12}^*\E,\pi_{23}^*\E)\rightarrow
\coker(g)\rightarrow \Im(f)\rightarrow 0, 
\]
we get the short exact sequence
\[
0\rightarrow V_1^* \LongRightArrowOf{f^*} \ker(g^*) \longrightarrow 
\left[\SheafExt^1_{\pi_{13}}(\pi_{12}^*\E,\pi_{23}^*\E)\right]^*
\rightarrow 0,
\]
by the vanishing (\ref{eq-Ext-1-Im-f-vanishes}) and equation
(\ref{eq-ker-g-*-is-reflexive}). Hence, 
$\left[\SheafExt^1_{\pi_{13}}(\pi_{12}^*\E,\pi_{23}^*\E)\right]^*$
is the middle sheaf cohomology of the complex dual to
(\ref{eq-locally-free-complex-representing-F}). 
The dual complex represents the object $\tau^*\F$, in the derived category,
so the middle sheaf cohomology is the pullback 
$\tau^*\SheafExt^1_{\pi_{13}}(\pi_{12}^*\E,\pi_{23}^*\E)$.
Reflexivity now follows, by applying the above argument to the
dual complex, since $\tau^2=id$.

The equality $\kappa(E)=\kappa(\F)+ch(\delta_*\StructureSheaf{\M})$ follows from Lemma \ref{lemma-partial-additivity}.
Now $(\delta_*\StructureSheaf{\M})^\vee=\delta_*\omega_{\M}[-2n]$, and $\omega_\M\cong \StructureSheaf{\M}$. 
Hence, $ch(\delta_*\StructureSheaf{\M})=ch((\delta_*\StructureSheaf{\M})^\vee)$.
Finally, we have
\[
\kappa(E^*)=\kappa(\tau^*E)=\tau^*(\kappa(E))=\tau^*[\kappa(\F)+ch(\delta_*\StructureSheaf{\M})]=
\kappa(\F^\vee)+ch((\delta_*\StructureSheaf{\M})^\vee)=\kappa(E^\vee).
\]


Part \ref{lemma-item-direct-images-of-V}, with $V$ replaced by $\widetilde{V}$, follows from Claim 
\ref{claim-higher-direct-images-of-U-s} and the long exact sequence of 
higher direct images via $\beta$, associated to the short exact sequence 
\[
0 \rightarrow U_{-1}\RightArrowOf{\tilde{g}} \ker(\tilde{f}) 
\rightarrow \widetilde{V} \rightarrow 0. 
\]

We prove next Part \ref{prop-item-restriction-of-V-is-a-symplectic-vb} with $V$ replaced by $\widetilde{V}$.
Assume first that the universal sheaf is untwisted.
Let $Z$ be the total space of the vector bundle $\Hom(V_{-1},V_0)$,
$h:Z\rightarrow \M\times \M$ the projection, 
$g':h^*V_{-1}\rightarrow h^*V_0$ the tautological homomorphism, 
$Z_1\subset Z$ the determinantal stratum, where the rank of $g'$ 
is $\rank(V_{-1})-1$, and 
$g:\M\times \M\rightarrow Z$ the section given in 
(\ref{eq-locally-free-complex-representing-F}). 
$Z_1$ is a smooth locally closed subvariety, whose normal bundle  $N_{Z_1}$ 
is isomorphic to  
$\Hom\left(\ker(\restricted{g'}{Z_{1}}),
\coker(\restricted{g'}{Z_{1}})\right)$ \cite{a-c-g-h}.
The diagonal $\Delta$ is the scheme theoretic inverse image $g^{-1}(Z_1)$.
Hence, the homomorphism
\begin{equation}
\label{eq-global-description-of-dg}
dg \ : \ N_{\Delta} \ \ \ \longrightarrow \ \ \ 
g^*N_{Z_1} \ = \ \Hom\left(\ker(\restricted{g}{\Delta}),
\coker(\restricted{g}{\Delta})\right)
\end{equation}
is injective at every fiber of $N_\Delta$.
$\Delta$ is also the degeneracy locus of the homomorphism $f$
given in (\ref{eq-locally-free-complex-representing-F}), and 
$f\circ g=0$. Thus, the image of $dg$ is contained in 
$\Hom\left(\ker(\restricted{g}{\Delta}),
\ker(\restricted{f}{\Delta})/\Im(\restricted{g}{\Delta})\right)$.
Now, $\ker(\restricted{g}{\Delta})\cong\StructureSheaf{\Delta}$
and $\ker(\restricted{f}{\Delta})/\Im(\restricted{g}{\Delta})$
is isomorphic to $T\Delta$, by the well known
identification of
$T\M$ with
the relative extension sheaf $\SheafExt^1_{f_2}(\E,\E)$. 
We conclude that $dg$ factors through a homomorphism
\[
dg \ : \ N_{\Delta} \ \ \ \longrightarrow \ \ \ T\Delta,
\]
which is fiber-wise injective, and hence an isomorphism.

The above argument is easily adapted to the case of a twisted universal sheaf as follows. The sheaves 
$\ker(\restricted{g}{\Delta})$ and $\coker(\restricted{g}{\Delta})$ are untwisted and so the description of $dg$, which is valid in each local chart, glues to the global description provided in Equation (\ref{eq-global-description-of-dg}). The rest of the argument is identical.

Over $B$ we have the tautological line-sub-bundle
$\eta \ : \ \StructureSheaf{D}(D)\hookrightarrow p^*N_{\Delta}$
and the homomorphism 
$d(\beta^*g)$ is the composition
$p^*(dg)\circ \eta$. It follows that the image of $d(\beta^*g)$
is $\ell\subset T\Delta$, by the definition of $\ell$. 
We claim, on the other hand, that the image of $d(\beta^*g)$ is precisely the line bundle
\[
\Hom\left(p^*\ker(\restricted{g}{\Delta}),
\Im(\restricted{\tilde{g}}{D})/\Im(\restricted{\beta^*g}{D})
\right).
\]
Observe first that the image of $\beta^*g$ is contained in
the image of $\tilde{g}$, which is a subbundle of $\beta^*V_0$.
We see that the image of $d(\beta^*g)$ is contained in the line bundle displayed above, by
repeating the above argument for the complex of locally free sheaves
\[
\beta^*V_{-1}\RightArrowOf{\beta^*g}\beta^*V_0\rightarrow\beta^*V_0/\Im(\tilde{g}). 
\]
The image is equal to the line bundle displayed above, since the latter is isomorphic to $\StructureSheaf{D}(D)$.
Indeed, we have observed already that $p^*\ker(\restricted{g}{\Delta})$ is the trivial line bundle and 
$\Im(\restricted{\tilde{g}}{D})/\Im(\restricted{\beta^*g}{D})$ is isomorphic to 
$U_{-1}/\beta^*V_{-1}\cong [\iota_*(p^*K)](D)$, which is isomorphic to 
$\StructureSheaf{D}(D)$.
These two descriptions of the image of $d(\beta^*g)$ provide a
canonical isomorphism $\ell\cong \Im(\restricted{\tilde{g}}{D})/\Im(\restricted{\beta^*g}{D})$.
We see that $\restricted{\widetilde{V}}{D}$ is a sub-bundle of $[p^*T\Delta]/\ell$.

Repeating the above argument, for the dual of the complex
(\ref{eq-locally-free-complex-representing-F}) and for
the homomorphism $f^*$, we get 
that $(\restricted{\widetilde{V}^*}{D})$ embedds as a subbundle of $[p^*T\Delta]/\ell$
as well (under the identification $T\Delta\cong T^*\Delta$, via
the symplectic structure). 
Composing the former embedding with the dual of the latter we see that $\restricted{\widetilde{V}}{D}$ is isomorphic to the image of a homomorphism from
$[(p^*T\Delta)/\ell]^*$ to $(p^*T\Delta)/\ell$. The composite homomorphism must be a non-zero multiple of the composition of the inclusion 
$\ell^\perp\rightarrow p^*T\Delta$ with the quotient homomorphism $p^*T\Delta\rightarrow (p^*T\Delta)/\ell$, as 
$\Hom(\ell^\perp,(\ell^\perp)^*)$ is one dimensional 
by Lemma \ref{lemma-hom-from-ell-perp-to-ell-perp-dual}.
Hence, $\restricted{\widetilde{V}}{D}$ is isomorphic to $\ell^\perp/\ell$. 

Part \ref{lemma-item-V-is-locally-free}) It suffices to prove the isomorphism $V\cong \widetilde{V}$, as we already know that $\widetilde{V}$ is locally free of rank $(v,v)$.
The direct image $p_*[(\widetilde{V}\restricted{)}{D}]$ vanishes, by part 
\ref{prop-item-restriction-of-V-is-a-symplectic-vb} and the vanishing
of $p_*[\ell^\perp/\ell]$. 
Hence, $\beta_*[\widetilde{V}(-D)]$ is isomorphic to $\beta_*\widetilde{V}$. We already established the isomorphism 
$\beta_*\widetilde{V}\cong E$ in the proof of part 
\ref{lemma-item-direct-images-of-V} (with $V$ replaced by $\widetilde{V}$). 
We get the isomorphism $\beta^*E\cong \beta^*\beta_*[\widetilde{V}(-D)]$.
The natural homomorphism $\beta^*\beta_*[\widetilde{V}(-D)]\rightarrow \widetilde{V}(-D)$ 
is surjective, by part 
\ref{prop-item-restriction-of-V-is-a-symplectic-vb}
and Lemma \ref{lem-facts-about-tautological-subquotient}.
The kernel of the composition 
$\beta^*E\cong \beta^*\beta_*[\widetilde{V}(-D)]\rightarrow \widetilde{V}(-D)$ is supported on $D$, and is hence 
the torsion subsheaf of $\beta^*E$.
The isomorphism $V(-D)\cong \widetilde{V}(-D)$ follows.

Part 
\ref{lemma-part-pull-back-of-V-is-its-dual})  
If we repeat the construction of the vector bundle $\widetilde{V}$ in Equation 
(\ref{eq-V-as-a-subquotient}), using the dual of the complex
(\ref{eq-locally-free-complex-representing-F}), we obtain the
vector bundle $\widetilde{V}^*$, by a direct check. 
On the other hand, the dual complex represents $\tau^*\F$, 
and the proof of the equality of the sheaves 
(\ref{eq-vector-bundle-over-B}) and (\ref{eq-V-as-a-subquotient})
yields the isomorphism
\[
\widetilde{V}^* \cong \beta^*\left[\tau^*\left\{
\SheafExt^1_{\pi_{13}}(\pi_{12}^*\E,\pi_{23}^*\E)
\right\}
\right](D)/tor.
\]
The statement now follows from the equality 
$\beta^*\tau^*=\tilde{\tau}^*\beta^*$.

Part \ref{lemma-item-tors})
Consider the exact triangle
\[
E \ \LongRightArrowOf{a} \ [V_{-1}\rightarrow V_0\rightarrow V_1] \ \LongRightArrowOf{b} \
\StructureSheaf{\Delta}[-1] \ \rightarrow \ E[1].
\]
Restriction to $\Delta$ yields the long exact sequence
\[
\begin{array}{cccccc}
\Tor_2^{\M\times\M}(E,\StructureSheaf{\Delta}) & \LongRightArrowOf{a_{-2}} & 
0 & \LongRightArrowOf{b_{-2}} & 
\Tor_3^{\M\times\M}(\StructureSheaf{\Delta},\StructureSheaf{\Delta})& 
\LongRightArrowOf{\delta_{-1}}
\\
\Tor_1^{\M\times\M}(E,\StructureSheaf{\Delta}) & \LongRightArrowOf{a_{-1}} & 
\StructureSheaf{\Delta} & \LongRightArrowOf{b_{-1}} & 
\Tor_2^{\M\times\M}(\StructureSheaf{\Delta},\StructureSheaf{\Delta})& 
\LongRightArrowOf{\delta_{0}}
\\
E\otimes \StructureSheaf{\Delta} & \LongRightArrowOf{a_{0}} & 
T\Delta & \LongRightArrowOf{b_{0}} & 
\Tor_1^{\M\times\M}(\StructureSheaf{\Delta},\StructureSheaf{\Delta})& 
\LongRightArrowOf{\delta_{1}}
\\
0 & \LongRightArrowOf{a_1} & 
\StructureSheaf{\Delta} &  \LongRightArrowOf{b_1} & 
\StructureSheaf{\Delta}\otimes \StructureSheaf{\Delta} \rightarrow 0.
\end{array}
\]
Note that $\Tor^{\M\times\M}_i(\StructureSheaf{\Delta},\StructureSheaf{\Delta})$
is isomorphic to $\Wedge{i}T^*\Delta$.
Clearly, $\delta_{-i}$ is an isomorphism, for $i\geq 2$. 
The isomorphism in Equation (\ref{eq-Tor-i-is-wedge-i+2}) follows for $i\geq 2$.
The homomorphism $b_0$ is surjective, hence an isomorphism.
Thus $a_0=0$ and $\delta_0$ is surjective. 

The isomorphisms in Equation (\ref{eq-Tor-i-is-wedge-i+2}) for $i=1$  and in Equation (\ref{eq-restriction-of-E-to-diagonal}) 
would both follow, once we prove that $b_{-1}$ is injective.
The proof is by contradiction. Assume that $b_{-1}$ vanishes. 
Then $\delta_0$ is injective and $\delta_0(\sigma)$ is a non-zero global section of
$H^0(E\otimes \StructureSheaf{\Delta})$.
Let $tor(\beta^*E)$ be the torsion subsheaf of $\beta^*E$. 
The endo-functor $R\beta_{*}L\beta^*$ of $D^b_{Coh}(\M\times\M,\pi_1^*(\theta^{-1})\pi_2^*\theta)$
is the identity. Hence, $\beta_*(tor(\beta^*E))=0$, since $E$ is torsion free,
by part \ref{lemma-item-SheafExt-1-is-reflexive}.
In particular, $H^0(tor(\beta^*E))=0$.
Now $[\beta^*E/tor(\beta^*E)\restricted{]}{D}\cong \ell^\perp/\ell$, by
part \ref{prop-item-restriction-of-V-is-a-symplectic-vb}, and $H^0(\ell^\perp/\ell)=0$.
Thus, $H^0(D,[\beta^*E\restricted{]}{D})=0$.
Consequently, $H^0(E\otimes \StructureSheaf{\Delta})=0$. 
A contradiction. This completes the proof of Proposition \ref{prop-V}.
\EndProof

\begin{rem}
The statement of Proposition \ref{prop-V} holds for smooth and projective moduli spaces of stable sheaves over an abelian surface.
The same proof applies with one exception, Lemma  \ref{lemma-hom-from-ell-perp-to-ell-perp-dual} is false in that case.
We sketch the  argument replacing the use of Lemma  \ref{lemma-hom-from-ell-perp-to-ell-perp-dual} in the proof of 
Proposition \ref{prop-V}, omitting the details. The same argument, used in the proof above,
identifies $p^*T\Delta$ with the degree $0$ sheaf cohomology 
$\ker(\beta^*\restricted{f}{D})/Im(\beta^*\restricted{g}{D})$
of the restriction to $D$ 
of the pullback $\beta^*V_\bullet$ of the complex (\ref{eq-locally-free-complex-representing-F}). Furthermore, 
a filtration 
\begin{equation}
\label{eq-filtration}
\ell\subset W \subset p^*T\Delta, 
\end{equation}
by subbundles of $p^*T\Delta$  is constructed, where $\ell$ is identified with
$Im(\restricted{\tilde{g}}{D})/Im(\restricted{\beta^*g}{D})$ and 
$W:=\ker(\restricted{\tilde{f}}{D})/Im(\beta^*\restricted{g}{D})$ is a co-rank $1$ subbundle of $p^*T\Delta$.
Lemma   \ref{lemma-hom-from-ell-perp-to-ell-perp-dual} was used to prove that $W$ is equal to the subbundle $\ell^\perp$ symplectic-orthogonal to $\ell$. Avoiding Lemma  \ref{lemma-hom-from-ell-perp-to-ell-perp-dual} one checks first that the filtration
(\ref{eq-filtration}) depends only on the object $\F$  in $D^b(\M(v)\times \M(v))$ represented by the complex (\ref{eq-locally-free-complex-representing-F}), so that any other quasi-isomorphic complex of locally free sheaves induces the same filtration of the degree $0$ cohomology of the restriction of its pullback to the exceptional divisor $D$. 
The complex $V_\bullet^*$ dual to (\ref{eq-locally-free-complex-representing-F})
yields an analogous filtration
\begin{equation}
\label{eq-second-filtration}
\mbox{ann}(W)\subset \mbox{ann}(\ell)\subset p^*T^*\Delta,
\end{equation}
where $\mbox{ann}(W)$ and $\mbox{ann}(\ell)$ are the subbundles annihilating $W$ and $\ell$. 
Now, as observed in the proof  above, the pullback $\tau^*(V_\bullet^*)$, by the transposition $\tau$ of the two factors,
is a locally free complex representing an object isomorphic to the object $\F$ represented by 
(\ref{eq-locally-free-complex-representing-F}). The induced isomorphism between the degree $0$ sheaf cohomologies
$p^*T^*\Delta=\H^0(\beta^*(\tau^*(V_\bullet^*)\restricted{)}{D})$ and $p^*T\Delta=\H^0(\restricted{(\beta^*V_\bullet)}{D})$
depends only on the choice of a trivialization of the canonical line-bundle of the abelian surface and corresponds to the canonical, up to a scalar factor, symplectic structure on $\M(v)$ constructed by Mukai \cite{mukai-symplectic-structure}. 
On the one hand the isomorphism maps the filtration (\ref{eq-filtration}) to (\ref{eq-second-filtration}), since the two complexes represent the same object. On the other hand, its symplectic interpretation implies that it maps $\ell^\perp$ to $\mbox{ann}(\ell)$. 
Hence,  $W=\ell^\perp$.
\end{rem}

\section{Lifting  deformations of a moduli space $\M$ to deformation of the pair $(\M,E)$}
Keep the notation of Section \ref{sec-a-reflexive-sheaf-E-and-its-resolution}. 
In particular, $\M:=\M_H(v)$ is a moduli space of stable sheaves over a $K3$ surface $S$ and 
$E$ is the reflexive sheaf over the product $\M\times \M$ introduced in  Proposition
\ref{prop-V}.
Let $S'$ be another projective $K3$ surface,
$v'\in K_{top}S'$ a primitive class satisfying 
$(v',v')=2n-2$, $n\geq 2$, 
and $H'$ a $v'$-generic ample line bundle.  
Assume that $\M':=\M_{H'}(v')$ is non-empty. 
Yoshioka proved that the moduli space $\M'$
is an irreducible holomorphic symplectic variety, deformation
equivalent to $S^{[n]}$ \cite{yoshioka-abelian-surface}. 
His proof implies the existence of a 
sequence of families of $K3$ surfaces $\S_i\rightarrow T_i$, $1\leq i\leq N$, 
over quasi-projective curves $T_i$,
with smooth and proper relative families of such moduli spaces $\M_{\S_i/T_i}$
having the following properties. 
There exist points $t'_i\in T_i$ and 
$t''_{i+1}\in T_{i+1}$, and an isomorphism 
$\phi_{i}$ from the fiber $\M_{t'_i}$  onto the fiber 
$\M_{t''_{i+1}}$. Finally, $\M_{t'_1}=\M_{H'}(v')$, and
$\M_{t''_N}=S^{[n]}$. 

The isomorphism $\phi_i$ comes in two flavors. One 
is induced by a Fourier-Mukai transformations
between the derived categories of $S_{t'_i}$ and $S_{t''_{i+1}}$
mapping stable sheaves to stable sheaves. Such Fourier-Mukai transformations
relate a twisted universal sheaf over $S_{t'_i}\times \M_{t'_i}$
to  one over $S_{t''_i}\times \M_{t''_i}$ 
\cite[Theorem 1.6]{mukai-applications}. 

The second flavor is induced by the composition, of a Fourier-Mukai 
transformation, with the functor, which takes an object or a morphism, in the 
derived category, to its dual.
The composite functor relates a twisted universal sheaf over 
$S_{t'_i}\times \M_{t'_i}$ 
to the {\em dual} of one over $S_{t''_{i+1}}\times \M_{t''_{i+1}}$ (see \cite[Theorem 7.9]{markman-monodromy-I} or 
\cite[Prop. 3.2]{yoshioka-abelian-surface}).

The following Lemma thus follows from Yoshioka's work.
Let $E$ be the twisted sheaf over $\M\times \M$ in Proposition 
\ref{prop-V} and $E'$ its analogue over 
$\M'\times \M'$. Note that $\SheafEnd(E)$ and $\SheafEnd(E^*)$
are isomorphic reflexive coherent sheaves, but they are 
not isomorphic as reflexive sheaves of Azumaya algebras.

\begin{lem}
\label{lem-lifting-to-deformations-of-pairs-ok-for-moduli-spaces}
The pair $(\M',\{\SheafEnd(E'),\SheafEnd((E')^*)\})$ deforms to the pair 
$(\M,\{\SheafEnd(E),\SheafEnd(E^*)\})$. 
The structures of Azumaya algebras deform as well.
\end{lem}

\section{Hyperholomorphic sheaves}
\label{sec-hyperholomorphic-sheaves}

We review
Verbitsky's theory of hyperholomorphic reflexive sheaves
\cite{kaledin-verbitski-book}. It plays a central role in the proof of Theorem 
\ref{thm-main-introduction}.

%
\subsection{Twistor deformations of pairs}
\label{sec-twistor-deformations-of-pairs}
Let $X$ be an irreducible holomorphic-symplectic manifold, $\omega$
a K\"{a}hler class of $X$, and $\X\rightarrow \PP^1_\omega$
the associated twistor deformation \cite{HKLR,huybrects-basic-results}. 
Recall that associated to $\omega$ and the complex structure $I$ 
is a Ricci-flat hermitian metric $g$, by the Calabi-Yau theorem 
\cite{beauville-varieties-with-zero-c-1}. 
Furthermore, any two among $I$, $\omega$, 
and $g$, determine the third.
The twistor deformation $\X\rightarrow \PP^1_\omega$
comes with a canonical differentiable trivialization $\X\cong X\times \PP^1_\omega$.
Let $\psi:\X\rightarrow X$ be the first projection.
The Riemannian metric on $X$ is constant with respect to this trivialization, 
but the complex structure $I_t$ and the associated K\"{a}hler form 
$\omega_t$ vary as we vary $t\in \PP^1_\omega$. 
We denote by $X_t$ the differentiable manifold $X$ endowed with
the complex structure $I_t$.
We denote by $0\in \PP^1_\omega$ the point corresponding to 
the complex structure $I$ on $X$.

Let $F$ be a reflexive sheaf on $X$ and $(F)_{sing}$ the singular locus of 
$F$. Then $(F)_{sing}$ has codimension $\geq 3$ in $X$.
Set $(F)_{sm}:=X\setminus (F)_{sing}$. 
Let $g_F$ be a hermitian metric on the restriction of $F$ to 
$(F)_{sm}$. Associated to $g_F$ and the holomorphic
structure $\bar{\partial}$ of $F$ is the Chern connection $\nabla$ 
\cite[Ch. 0 Sec. 5, Lemma page 73]{GH}. 
Recall that 
$\bar{\partial}$ is the $(0,1)$-part of $\nabla$.
The decomposition $T^*_{\ComplexNumbers}X:=T^{1,0}X\oplus T^{0,1}X$, 
of the complexified cotangent bundle of $X$,
depends on the complex structure $I$ of $X$. 

When the sheaf $F$ is $\omega$-slope-stable, then there exists a unique
Hermite-Einstein metric $g_F$, whose curvature form is $L^2$-integrable,
on the restriction of $F$ to $(F)_{sm}$ \cite{bando-siu}. 
We will refer to $g_F$ as the {\em Hermite-Einstein metric}
of $F$ and to its Chern connection as the {\em Hermite-Einstein connection}
of $F$. Denote by $\bar{\partial}_t$, $t\in  \PP^1_\omega$,
the $(0,1)$-part of $\nabla$ with respect to the complex structure $I_t$. 
Then $\bar{\partial}_0^2=0$, but $\bar{\partial}_t^2$ need not vanish for 
a general $t\in \PP^1_\omega$.

\begin{defi}
\label{def-hyperholomorphic-sheaf}
\cite[Def. 3.15]{kaledin-verbitski-book}
An $\omega$-slope-stable reflexive sheaf $F$ over $(X,\omega)$ is 
{\em $\omega$-stable-hyperholomorphic}, if $\bar{\partial}_t^2=0$, 
for all $t\in \PP^1_\omega$.
An $\omega$-slope-polystable reflexive sheaf 
$F$ is {\em $\omega$-polystable-hyperholomorphic}, 
if each $\omega$-slope-stable direct summand of $F$ 
is $\omega$-stable-hyperholomorphic.
\end{defi}

The $\omega$-slope of an $\omega$-polystable-hyperholomorphic reflexive sheaf is zero, by \cite[Rem. 3.12]{kaledin-verbitski-book}.
\begin{rem}
\label{rem-su-invariance}
Note that the condition $\bar{\partial}_t^2=0$, 
for all $t\in \PP^1_\omega$, in the above definition is equivalent to the $SU(2)$-invariance of the curvature of $\nabla$ appearing in
\cite[Def. 3.15]{kaledin-verbitski-book}. The equivalence follows from \cite[Lemma 2.6]{kaledin-verbitski-book}.
\end{rem}

\begin{defi} 
\label{def-trianalytic}
\cite[Def. 2.9]{kaledin-verbitski-book}
A subvariety $Z$ of $X$ is {\em $\omega$-tri-analytic}, if the canonical 
differentiable trivialization $\X\cong X\times \PP^1_\omega$
maps $Z\times \PP^1_\omega$ to a closed analytic subvariety of $\X$.
\end{defi}

Verbitsky proves that the singularity locus $Z:=(F)_{sing}$, of a reflexive 
hyperholomorphic sheaf, is supported over a 
tri-analytic\footnote{Note that $Z$ is tri-analytic if and only if
$Z$ is analytic with respect to $I_t$, for all $t\in \PP^1_\omega$. The `only if'
direction is clear. The `if' direction follows from the following fact.
Given points $t\in\PP^1_\omega$ and $x\in X$, we get the direct sum decomposition
$T_{(x,t)}\X=T_t\PP^1_\omega\oplus T_xX$, of the real tangent space, induced by the
differentiable trivialization of $\X$. The relevant fact is that both summands are 
complex subspaces, even though the 
projection $\X\rightarrow X$ 
is not holomorphic \cite[formula (3.71)]{HKLR}.}
subvariety of $X$
\cite[Claim 3.16]{kaledin-verbitski-book}. 
The complex structure $\bar{\partial}_t$ on $F$ defines a locally free 
$\StructureSheaf{X_t}$-module over $X_t\setminus Z$.
We denote by $F_t$ the reflexive sheaf on $X_t$ corresponding to the 
push-forward of the latter via the inclusion into $X_t$. In particular, $F_0=F$.
The pushforward  $F_t$ is a reflexive coherent sheaf, by the Main Theorem of \cite{siu}, since the complex codimention of $(F)_{sing}$ is $\geq 3$. The sheaf $F_t$  is $\omega_t$-polystable, by \cite[Prop. 3.17]{kaledin-verbitski-book}.
The pullback $\psi^*(\restricted{F}{X\setminus Z})$ to $\X\setminus [Z\times \PP^1_\omega]$, of the vector bundle associated to the restriction of $F$ to its locally free locus, is endowed with the pulled back connection, whose $(0,1)$ part is
an integrable complex structure, by \cite[Lemma 5.1]{kaledin-verbitsky-non-hermetian}. Its push-forward to $\X$ is a reflexive coherent sheaf $\F$, by the Main Theorem of \cite{siu}.   $F_t$ is the reflexive hull of the quotient of the restriction of $\F$ to the fiber of $\X$ over $t\in \PP^1_\omega$ by its torsion subsheaf, as the two agree away from $Z$. 
The following is a fundamental result of Verbitsky.

\begin{thm}
\label{thm-hyperholomorphic-iff-slope-stable}
\cite[Theorem 3.19]{kaledin-verbitski-book}
Let $E$ be an $\omega$-slope-stable reflexive sheaf on $X$.
Assume that $c_i(E)$ is of Hodge type $(i,i)$, for $i=1,2$, 
and for all complex structures parametrized by the twistor line $\PP^1_\omega$. 
Then $E$ is $\omega$-stable hyperholomorphic. 
\end{thm}

The notion of  $\omega$-slope-stability is well defined
for twisted sheaves as well. 
Slope-stability 
of a torsion-free sheaf $E$ depends on the sheaf $\SheafEnd(E)$
of Lie-algebras and its subsheaves of maximal parabolic subalgebras. Given a subsheaf $F$ of $E$, 
the condition $\slope_\omega(F)<\slope_\omega(E)$ is equivalent to 
\begin{equation}
\label{eq-slope-inequality}
\deg_\omega(\SheafHom(E,F)) \ \ < \ \ 0. 
\end{equation}
The sheaf $\SheafHom(E,F)$ is untwisted, 
for every $\theta$-twisted subsheaf $F$ of a $\theta$-twisted sheaf $E$. 

\begin{defi}
\label{def-slope-stability}
\begin{enumerate}
\item
Let $E$ be a torsion free  $\theta$-twisted sheaf  and 
$\omega$ a K\"{a}hler class on $X$.
We say that $E$ is {\em $\omega$-slope-stable}, if the inequality
(\ref{eq-slope-inequality}) holds, for every non-zero 
$\theta$-twisted proper subsheaf $F$ of $E$. The sheaf $E$ is 
{\em $\omega$-slope-semistable}, if the analogue of 
(\ref{eq-slope-inequality}),
with strict inequality replaced by $\leq$, holds for every such $F$.
The sheaf $E$ is said to be {\em $\omega$-slope-polystable},
if it is $\omega$-slope-semistable and 
away from a locus of codimension two $E$ is isomorphic to a direct sum
of $\omega$-slope-stable sheaves. 
\item
A reflexive sheaf $A$ of Azumaya algebras (Definition \ref{def-Azumaya})
is {\em $\omega$-slope-stable}
(resp. {\em $\omega$-slope-polystable}), if some, hence any lift of $A$ to a twisted reflexive sheaf
has the corresponding property.
Equivalently,\footnote{Given a subspace $W$ of a vector space $V$ we get the maximal parabolic subalgebra of $\LieAlg{gl}(V)$ consisting of endomorphisms of $V$ which map $W$ to itself. All maximal parabolic subalgebras of $\LieAlg{gl}(V)$ are obtained that way. A subsheaf $P$ of maximal parabolic subalgebras of a reflexive sheaf $A$ of Azumaya algebras is a subsheaf, which away from the singularities of $A$ corresponds to a subbundle of maximal parabolic subalgebras in each fiber. If $A=\SheafEnd(E)$ and $P$ corresponds to a subsheaf $F$ of $E$, then 
$\deg_\omega(P)=\deg_\omega(\SheafHom(E,F))$, since $\deg_\omega(P/\SheafHom(E,F))=\deg_\omega(\SheafEnd(E/F))=0.$
}  
$A$ is $\omega$-slope-stable, if every non-trivial 
subsheaf of maximal parabolic subalgebras of $A$
has negative $\omega$-slope.
\end{enumerate}
\end{defi}


Note that if $E$ is reflexive and $\omega$-slope-polystable, then
$E$ is a direct sum of $\omega$-slope-stable sheaves
\cite[Cor. 1.6.11]{huybrechts-lehn-book}.

\begin{prop}
\label{prop-polystability-of-a-very-twisted-sheaf}
Let $(X,\omega)$ be a compact K\"{a}hler manifold and $E$ a reflexive $\omega$-slope-stable $\theta$-twisted coherent sheaf, for some $\theta\in H^2_{an}(X,\StructureSheaf{}^*)$.
Then the sheaf $\SheafEnd(E)$  is $\omega$-slope-polystable. 
\end{prop}

The proposition is proven in the Appendix Section \ref{sec-semistability}. The untwisted case of the proposition is well known and follows from \cite[Theorem 3]{bando-siu} stating that the existence of an admissible $\omega$-Hermite-Einstein metric 
on $E$ is equivalent to $\omega$-slope-polystability of $E$.

\begin{lem} \cite[Section 3.5]{kaledin-verbitski-book}
\label{lemma-extension-of-a-slope-zero-subsheaf-to-a-hyperholomorphic-subsheaf}
Let $F$ and $G$ be two reflexive $\omega$-polystable-hyperholomorphic sheaves of
$\omega$-slope $0$.  Then the following statements hold.
\begin{enumerate}
\item
\label{lemma-item-global-sections-are-flat}
Any global section $f$ of $F$ is flat with respect to the 
Hermite-Einstein connection. In particular, $f$ is a holomorphic section
with respect to all complex structures $\bar{\partial}_t$, $t\in \PP^1_\omega$.
\item
\label{lemma-item-parallel-transport-along-twistor-line}
There exists a canonical isomorphism of vector spaces
$\Hom(F_t,G_t)\rightarrow \Hom(F_s,G_s)$, for all $s,t\in \PP^1_\omega$.
\item
\label{lemma-item-algebraic-structures}
If $F_t$ is endowed with an associative multiplication
$m_t:F_t\otimes F_t\rightarrow F_t$, 
or more specifically a structure of an Azumaya $\StructureSheaf{X_t}$-algebra,
or a Lie-algebra structure 
$[,]_t:F_t\otimes F_t\rightarrow F_t$, then $F_s$ is naturally
endowed with such a structure, for all $s\in \PP^1_\omega$.
\item
\label{lemma-item-slope-zero-subsheaf-is-hyperholomorphic}
Any saturated subsheaf $F'$ of $F$ of $\omega$-slope zero is 
reflexive and $\omega$-polystable-hyperholomorphic. 
\item
\label{lemma-item-ker-and-Im-are-hyperholomorphic}
Let $\varphi:F\rightarrow G$ be a homomorphism.
Then $\ker(\varphi)$ and $\mbox{Im}(\varphi)$ are 
$\omega$-polystable-hyperholomorphic. 
\item
\label{lemma-item-slope-zero-subsheaf-extends}
Let $F'_t$ be a saturated subsheaf of $F_t$, of $\omega_t$-slope $0$, 
for some $t\in \PP^1_\omega$.
Then $F'_t$ extends to an $\omega$-polystable-hyperholomorphic subsheaf 
$F'_s$ of $F_s$, for all $s\in \PP^1_\omega$.
\item
\label{lemma-item-extension-of-parabolic-is-parabolic}
If $F_t$ has a structure of an Azumaya  algebra and the subsheaf $F'_t$ in part 
\ref{lemma-item-slope-zero-subsheaf-extends} is a maximal parabolic subalgebra, 
then the subsheaf $F'_s$ is a maximal parabolic subalgebra, for all $s\in \PP^1_\omega$.
\end{enumerate}
\end{lem}

\begin{proof}
Part \ref{lemma-item-global-sections-are-flat}) A global section of $F$ corresponds to a direct summand of $F$ isomorphic to the trivial line bundle, since $F$ is polystable of $\omega$-slope zero. The statement reduces to the case of a trivial line bundle, in which it is clear. Note also that the $(0,1)$ part of the connection with respect to the complex structure $I$ of $X$ is the $(1,0)$ part with respect to the conjugate complex structure $-I$, so if a section is holomorphic with respect to both $I$ and $-I$, then it is flat. 

Part \ref{lemma-item-parallel-transport-along-twistor-line})
Follows from Part \ref{lemma-item-global-sections-are-flat}. See \cite[Theorem 3.27]{kaledin-verbitski-book}.

Part \ref{lemma-item-algebraic-structures}) 
The sheaf $\SheafHom\left(\SheafHom(F_t^*,F_t),F_t\right)$
is $\omega_t$-polystable-hyperholomorphic and $m$ (or $[,]$)
is a global section of this sheaf, hence a flat section with respect to the induced
Hermite-Einstein connection, hence a
holomorphic section with respect to all induced complex structures, by 
part \ref{lemma-item-global-sections-are-flat}. 
The axioms of the corresponding algebraic structure are 
expressed as identities involving flat sections. Hence they hold
for all $s\in \PP^1_\omega$, since they hold at $t$.

Part \ref{lemma-item-slope-zero-subsheaf-is-hyperholomorphic}) 
The $\omega$-slope-stable summands of $F$ are hyperholomorphic,
and $F'$ is necessarily isomorphic to a direct sum of such summands.

Part \ref{lemma-item-ker-and-Im-are-hyperholomorphic}) The kernel and image of
$\varphi$ must have $\omega$-slope zero. It remains to prove that the image is a saturated subsheaf.
Now $\varphi$ factors through an injective homomorphism from a direct summand $F'$ of $F$ into $G$, and $F'$ is reflexive, 
since $F$ is.
$Im(\varphi)$ is thus reflexive and its saturation in $G$ has the same slope as $Im(\varphi)$ and so $Im(\varphi)$ is already saturated.

Part \ref{lemma-item-slope-zero-subsheaf-extends})
$F_t$ is $\omega_t$-slope-polystable hyperholomorphic, for all $t\in \PP^1_\omega$, by \cite[Prop. 3.17 and Theorem 3.19]{kaledin-verbitski-book}.
The sheaf $F'_t$ is $\omega_t$-slope-polystable hyperholomorphic,
by part \ref{lemma-item-slope-zero-subsheaf-is-hyperholomorphic}. It is furthermore a direct summand of $F_t$, as the latter is polystable of the same slope. The statement thus follows from the uniqueness of the hyperholomorphic connection
\cite[Remark 3.20]{kaledin-verbitski-book} and its compatibility with the direct sum decomposition.

Part \ref{lemma-item-extension-of-parabolic-is-parabolic}) 
Assume that $F'_t$ is a saturated $\omega_t$-slope $0$ Lie subalgebra.
Then its extension, in part \ref{lemma-item-slope-zero-subsheaf-extends},
consists of Lie subalgebras $F'_s$, for all $s\in \PP^1_\omega$, by part
\ref{lemma-item-algebraic-structures}.
Let $F_s''$ be the subsheaf orthogonal to $F'_s$ with respect to the trace bilinear pairing. 
Then $F'_s$ is a sheaf of 
maximal parabolic subalgebra over $(F)_{sm}$,
if and only if the following two conditions hold:
a) $F_s''$ is a subsheaf of $F'_s$, and 
b) the homomorphism
$F_s''\otimes F_s''\rightarrow F_s$, given by $a\otimes b\mapsto ab$, vanishes.
Now $F_s''$ is the kernel of the homomorphism $F'_s\rightarrow (F'_s)^*$,
induced by the trace pairing.
Hence $F_s''$ is saturated of $\omega_s$-slope $0$, and so 
$\omega_s$-polystable-hyperholomorphic, by part
\ref{lemma-item-slope-zero-subsheaf-is-hyperholomorphic}. 

Assume now that $F'_t$ is a saturated $\omega_t$-slope $0$ 
maximal parabolic subalgebra of $F_t$. We get a flag $F''\subset F'\subset F$
of $\omega$-polystable-hyperholomorphic reflexive sheaves of $\omega$-slope $0$.
Furthermore, each of the conditions a) and b) above  
is expressed in terms of the vanishing of
a natural homomorphism between $\omega$-polystable-hyperholomorphic
sheaves of slope zero. Hence, if they both hold for $t$, 
then they both hold for all $s\in \PP^1_\omega$,
by part \ref{lemma-item-global-sections-are-flat}.
\end{proof}

%
\subsection{Projectively $\omega$-stable-hyperholomorphic sheaves}

\begin{defi}
\label{def-projectively-omega-stable-hyperholomorphic}
Let $F$ be a reflexive $\omega$-slope-stable (possibly twisted) 
sheaf of positive rank over $(X,\omega)$. 
We say that $F$ is 
{\em projectively $\omega$-stable-hyperholomorphic} 
if, in addition, the sheaf $\SheafEnd(F)$ is $\omega$-polystable-hyperholomorphic.
A reflexive $\omega$-slope-polystable sheaf is 
{\em projectively $\omega$-polystable-hyperholomorphic}, 
if it is a direct sum of projectively $\omega$-stable-hyperholomorphic sheaves.
\end{defi}

\hide{
\begin{rem}
The two form $\lambda_t$ above is $\bar{\partial}$-closed and 
$\left(\frac{r\sqrt{-1}}{2\pi}\right)\lambda_t$, with $r=\rank(F)$, represents 
the projection of $c_1(F)$ to $H^{0,2}(X_t)$, since 
$\bar{\partial}_t^2$ is the $(0,2)$-part of the curvature form of $\nabla$ with respect 
to the complex structure $I_t$ of $X_t$.
\end{rem}
}

Let $F$ be a projectively $\omega$-polystable-hyperholomorphic reflexive (possibly twisted) sheaf
of rank $r>0$. As the singular locus of $F$ has codimension $\geq 3$, we have the isomorphism 
$H^2(X,\mu_r)\cong H^2(X\setminus (F)_{sing},\mu_r)$. We get the characteristic class $\tilde{\theta}\in H^2(X,\mu_r)$ of the projective bundle associated to $F$ over $X\setminus (F)_{sing}$
via Equation (\ref{eq-connecting-homomorphism-theta-tilde}). 
If $F$ happens to be untwisted, this class is 
$\exp\left(-2\pi\sqrt{-1} c_1(F)/r\right)$, as in equation
(\ref{eq-tilde-theta-via-c-1}). 
Denote by $\theta_t$ the image of 
$\tilde{\theta}$ in $H^2_{an}(X_t,\StructureSheaf{X_t}^*)$. Similarly,
let $\theta$ be the image of $\tilde{\theta}$ in $H^2_{an}(\X,\StructureSheaf{\X}^*)$ via
the composite homomorphism 
\begin{equation}
\label{eq-from-H-2-mu-r-of-fiber-to-Brauer-group-of-twistor-space}
H^2(X,\mu_r)\RightArrowOf{\psi^*} H^2(\X,\mu_r)\rightarrow H^2_{an}(\X,\StructureSheaf{\X}^*),
\end{equation}
where the left homomorphism is the pull-back via the 
projection $\psi:\X\rightarrow X$, associated to the 
differentiable trivialization of the twistor deformation.

\begin{construction}
\label{construction-family-of-projectively-hyperholomorphic-sheaves}
The sheaf $F$ corresponds to a reflexive sheaf $\A$ of Azumaya algebras 
(Definition \ref{def-Azumaya}) 
with Brauer class $\theta$ over the twistor space $\X$.
Following is the construction of such a family.
The sheaf $\SheafEnd(F)$ is $\omega$-polystable-hyperholomorphic, by assumption.
Hence $\SheafEnd(F)$ extends to a reflexive sheaf $\A$ over $\X$. 
The structure on $\SheafEnd(F)$ of 
a reflexive sheaf of Azumaya algebras extends to one on $\A$, by
Lemma \ref{lemma-extension-of-a-slope-zero-subsheaf-to-a-hyperholomorphic-subsheaf}
part (\ref{lemma-item-algebraic-structures}). It remains to prove that the Brauer class of
$\A$ is $\theta$. Now $\A$ has rank $r$ and thus determines a class $\alpha$ in
$H^2(\X,\mu_r)$ (use the homomorphism 
(\ref{eq-connecting-homomorphism-theta-tilde}) 
and the fact that the singular locus of $\A$ 
has codimension $\geq 3$ in $\X$). The class $\alpha$ restricts to the class $\tilde{\theta}$ in
$H^2(X,\mu_r)$. Hence, it suffices to prove that the image of the composite homomorphism 
(\ref{eq-from-H-2-mu-r-of-fiber-to-Brauer-group-of-twistor-space}) is equal to the $r$-torsion subgroup.
Now $H^2(\X,\mu_r)$ is isomorphic to $H^2(\PP^1_\omega,\mu_r)\oplus H^2(X,\mu_r)$ and the image of the 
summand $H^2(\PP^1_\omega,\mu_r)$ in $H^2_{an}(\X,\StructureSheaf{\X}^*)$ is trivial, as it is already trivial
in $H^2_{an}(\PP^1_\omega,\StructureSheaf{\PP^1_\omega}^*)$.
\end{construction}

Denote by $Z$ the singular locus of $\A$ and let $Z_t$ be its fiber over $t\in \PP^1_\omega$.
We keep the convention of Section \ref{sec-twistor-deformations-of-pairs} and denote by $\A_t$ the 
pushforward to $X_t$ of the restriction of $\A$ to $X_t\setminus Z_t$. Then $\A_t$ is a coherent reflexive sheaf,
by the Main Theorem of \cite{siu}.

\begin{lem}
\label{lemma-if-stable-for-some-t-then-stable-for-all-t}
Let $F$ be a reflexive projectively $\omega$-polystable-hyperholomorphic twisted 
sheaf. Let $\A$ be the reflexive sheaf of Azumaya algebras over the twistor family $\X$
associated to $\SheafEnd(F)$ via 
Construction \ref{construction-family-of-projectively-hyperholomorphic-sheaves}.
If $\A_t$ is an $\omega_t$-slope-stable sheaf of Azumaya algebras over $X_t$
for some $t$, then it is $\omega_t$-slope-stable for every $t$.
\end{lem}

\begin{proof}
Assume that $\A'_t$ is a saturated subsheaf of $\A_t$ of
maximal parabolic subalgebras, and $\A'_t$ has $\omega_t$-slope $0$,
for some $t\in \PP^1_\omega$. Then $\A'_t$ extends as an $\omega$-polystable-hyperholomorphic
subsheaf $P$ of $\A$,  with $P_0$ an $\omega$-slope $0$ subsheaf of $\SheafEnd(F)$, by Lemma 
\ref{lemma-extension-of-a-slope-zero-subsheaf-to-a-hyperholomorphic-subsheaf} part
\ref{lemma-item-slope-zero-subsheaf-extends}.
The subsheaf $P_t=\A'_t$ is a sheaf of maximal parabolic subalgebras. 
Its extension is also a subsheaf of maximal  parabolic subalgebras,
by Lemma 
\ref{lemma-extension-of-a-slope-zero-subsheaf-to-a-hyperholomorphic-subsheaf}
part \ref{lemma-item-extension-of-parabolic-is-parabolic}.
In particular, if $P_t$ is a non-zero proper subsheaf, then $\A_t$ is not
$\omega_t$-slope-stable, for any $t$.
\end{proof}

\begin{thm}
\label{thm-Verbitskys-inequality}
\cite[Cor. 3.24]{kaledin-verbitski-book}
Let $F$ be an $\omega$-slope-polystable reflexive sheaf on $(X,\omega)$,
such that  $c_1(F)/\rank(F)=c_1(F')/\rank(F')$ for every direct summand $F'$ of $F$.
Let $I_t$ be an induced complex structure 
such that $I_t\not\in\{I,-I\}$. Then 
\begin{equation}
\label{eq-verbitskys-inequality}
\int_X\kappa_2(F)\omega^{2n-2} \ \ \ \geq \ \ \ 
\Abs{\int_X\kappa_2(F)\omega_t^{2n-2}},
\end{equation}
and equality holds, if and only if each stable direct summand $F'$
of $F$ is $\omega$-stable-hyperholomorphic. Furthermore, equality holds
in (\ref{eq-verbitskys-inequality}) if $\kappa_2(F)$ is of Hodge type $(2,2)$
with respect to $I_t$, for all $t\in \PP^1_\omega$.
\end{thm}

\begin{proof}
When $F$ is $\omega$-slope-stable this is precisely 
\cite[Cor. 3.24]{kaledin-verbitski-book}. 
I thank Misha Verbitsky for pointing out this statement and the fact 
that the statement holds also when $F$ is $\omega$-slope-polystable.
Assume that $F=\oplus_{i=1}^N F_i$, where $F_i$ 
is $\omega$-slope-stable. Set $r:=\rank(F)$ and $r_i:=\rank(F_i)$.
Then 
\[
\kappa_2(F)=ch_2(F)-\frac{r}{2}(c_1(F)/r)^2=\sum_{i=1}^N [ch_2(F_i)-\frac{r_i}{2}(c_1(F_i)/r_i)^2]=
\sum_{i=1}^N\kappa_2(F_i).
\]
We get 
\[
\int_X\kappa_2(F)\omega^{2n-2}=\sum_{i=1}^N\int_X\kappa_2(F_i)\omega^{2n-2}
\geq \sum_{i=1}^N\Abs{\int_X\kappa_2(F_i)\omega_t^{2n-2}}
\geq \Abs{\int_X\kappa_2(F)\omega_t^{2n-2}},
\]
where the first inequality is by \cite[Cor. 3.24]{kaledin-verbitski-book},
and the second by the triangle inequality.
Clearly, equality holds above, if and only if it holds for each $F_i$.

If $\kappa_2(F)$ is of Hodge type $(2,2)$ with respect to $I_t$, 
for all $t\in \PP^1_\omega$, then equality holds in 
(\ref{eq-verbitskys-inequality}),
by  \cite[Claim 3.21]{kaledin-verbitski-book} and Remark \ref{rem-su-invariance} above.
\end{proof}

The following generalization of Theorem 
\ref{thm-hyperholomorphic-iff-slope-stable} was explained to me by Misha 
Verbitsky. 

\begin{cor}
\label{cor-main-result-on-projectively-hyperholomorphic-reflexive-sheaves}
\begin{enumerate}
\item
\label{cor-item-untwisted}
Let $E$ be an $\omega$-slope-stable (possibly twisted)
reflexive sheaf. Assume that $\kappa_2(E)$
remains of Hodge type $(2,2)$ along the chosen twistor line 
and the first Chern class of every direct summand of $\SheafEnd(E)$ vanishes. 
Then $\SheafEnd(E)$ is 
$\omega$-polystable-hyperholomorphic
and $E$ is projectively $\omega$-stable-hyperholomorphic.
\item
\label{cor-item-twisted}
Let $A$ be an $\omega$-slope-stable reflexive sheaf of Azumaya algebras of rank $r$
(Definition \ref{def-slope-stability}). 
Assume that $c_2(A)$ 
remains of Hodge type $(2,2)$ along the chosen twistor line 
and the first Chern class of every direct summand of $A$ vanishes. 
Then $A$ extends to a reflexive sheaf $\A$ 
of Azumaya algebras over $\X$, 
and $\A_t$ is a $\omega_t$-slope-stable sheaf of Azumaya algebras, for all $t\in \PP^1_\omega$.
\end{enumerate}
\end{cor}

\begin{proof}
(\ref{cor-item-untwisted})
The sheaf $\SheafEnd(E)$ is $\omega$-slope-polystable,
by Proposition \ref{prop-polystability-of-a-very-twisted-sheaf}.
Apply Theorem \ref{thm-Verbitskys-inequality} with $F:=\SheafEnd(E)$ to conclude that 
$\SheafEnd(E)$ is $\omega$-polystable-hyperholomorphic.

(\ref{cor-item-twisted}) $A$ is isomorphic to $A^*$ as a coherent sheaf, using
the trace bilinear pairing, and thus $c_1(A)=0.$
The construction of $\A$ follows from Theorem \ref{thm-Verbitskys-inequality}.
The structure of Azumaya algebra extends, by Lemma 
\ref{lemma-extension-of-a-slope-zero-subsheaf-to-a-hyperholomorphic-subsheaf}
part (\ref{lemma-item-algebraic-structures}). The stability of $\A_t$ follows from Lemma 
\ref{lemma-if-stable-for-some-t-then-stable-for-all-t}.
\end{proof}

\begin{rem}
\label{rem-caution-not-flat}
The extension $\A$ in Corollary \ref{cor-main-result-on-projectively-hyperholomorphic-reflexive-sheaves} is not known to be flat over $\PP^1_\omega$. We know only that its singular locus is tri-analytic, so the dimension of the intersection of the singular locus with the fibers of the twistor family is constant \cite[Claim 3.16]{kaledin-verbitski-book}.
\end{rem}

%
\subsection{Deformation of pairs along twistor paths}
\label{sec-deformations-along-twistor-paths}

A {\em marking} of an irreducible
holomorphic symplectic manifold $X$ is an isometry 
$\phi:H^2(X,\Integers)\rightarrow \Lambda$  with a fixed lattice $\Lambda$.
Let $\mathfrak{M}_\Lambda$ be the moduli space of isomorphism classes of
marked irreducible holomorphic symplectic manifolds 
\cite{huybrects-basic-results}. 
A {\em twistor path} in $\mathfrak{M}_\Lambda$ is a sequence of
twistor lines, in which each consecutive pair has non-trivial
intersection in $\mathfrak{M}_\Lambda$, together with a choice of an intersection point for each consecutive pair. 
If the chosen intersection point, of each consecutive pair,
corresponds to a manifold with trivial Picard group, 
we call the twistor path {\em generic}. 

\begin{thm}
\label{thm-path-connectedness}
\cite[Theorems 3.2 and 5.2.e]{verbitsky-cohomology}
Let $(X_i,\phi_i)$, $i=1,2$, be two marked irreducible
holomorphic symplectic manifolds, in the same connected component of 
$\mathfrak{M}_\Lambda$. 
Then there exists a generic twistor path in $\mathfrak{M}_\Lambda$ 
connecting $(X_1,\phi_1)$ with $(X_2,\phi_2)$. 
\end{thm}

We will need the following evident lemma.

\begin{lem}
\label{lemma-slope-stability-does-not-depend-on-kahler-class}
Let $X$ be a compact K\"{a}hler manifold with a trivial Picard group
$\Pic(X)=\{\StructureSheaf{X}\}$, $\omega$ and $\omega'$ two 
K\"{a}hler classes on $X$, and $E$ a torsion free, possibly  twisted, coherent 
$\StructureSheaf{X}$-module of rank $r$.
Then $E$ is $\omega$-slope-stable, if and only if $E$ does not admit
any subsheaf of rank $r'$, for $0<r'<r$. In particular, 
$E$ is $\omega$-slope-stable, if and only if $E$ is 
$\omega'$-slope-stable.
\end{lem}


A {\em parametrized twistor path} $\gamma:C\rightarrow \mathfrak{M}_\Lambda$
consists of a connected reduced nodal curve $C$, 
of arithmetic genus $0$, with an 
ordering of the irreducible components, 
so that two consecutive components meet at a node, 
and a morphism $\gamma$ from $C$ to $\mathfrak{M}_\Lambda$, mapping the
$i$-th component of $C$ isomorphically onto a twistor line.
If $\gamma$ maps each node to  a point with a trivial Picard group, we call $\gamma$ 
a {\em generic parametrized twistor path}. 
Let $\gamma:C\rightarrow \mathfrak{M}_\Lambda$ be a parametrized twistor path, 
$\X\rightarrow C$ the natural twistor deformation, 
$0\in C$ a point of the first component of $C$, and
$X_0$ the fiber of $\X$ over $0$. 
Let $E$ be a reflexive twisted sheaf on $X_0$. 

\begin{defi}
\label{def-can-be-deformed-along-a-twistor-path}
\begin{enumerate}
\item
Let $\E$ be a reflexive sheaf over $\X$, whose singular locus  is equidimensional over $C$ of codimension $\geq 3$.
The {\em reflexive restriction} of $\E$ to the fiber $X_t$ of $\X$ over $t\in C$ is the convex hull of the quotient of the restriction of
$\E$ to $X_t$ by its torsion subsheaf.
\item
We say that $E$ 
{\em can be deformed along $\gamma$}, if there exists 
a reflexive twisted coherent sheaf $\E$ over $\X$, such that the singular locus of $\E$ is equidimensional\footnote{We do not require $\E$ to be flat over $C$.} over $C$  of codimension $\geq 3$,
which reflexive restriction to $X_0$
represents the equivalence class of $E$ (Definition \ref{def-equivalent-twisted-sheaves}).
Equivalently, there exists a  reflexive sheaf of
Azumaya $\StructureSheaf{\X}$-algebras, with such a singular locus, 
which reflexive restriction to $X_0$ is isomorphic to $\SheafEnd(E)$. 
\end{enumerate}
\end{defi}

Let $X$ be an irreducible holomorphic symplectic manifold and 
$\gamma:C\rightarrow\mathfrak{M}_\Lambda$ a generic parametrized twistor path,
with $X_0=X$. Let $\omega_0$ be a K\"{a}hler class on $X_0$,
such that $\PP^1_{\omega_0}$ is the first twistor line. Let 
$\omega_{t_i^-}$, $1\leq i\leq N$, be a K\"{a}hler class on the 
$i$-th node $X_{t_i}$,
such that $\PP^1_{\omega_{t_i^-}}$ is the $i$-th twistor line,
and $\omega_{t_i^+}$, $1\leq i\leq N-1$, a K\"{a}hler class on $X_{t_i}$,
such that $\PP^1_{\omega_{t_i^+}}$ is the $i+1$ twistor line. 
Note that $\omega_0$ determines $\omega_{t_1^-}$, and 
$\omega_{t_i^+}$ determines $\omega_{t_{i+1}^-}$. 
At a node $t_i\in C$, the group $\Pic(X_{t_i})$ is trivial. Slope-stability 
is then independent of the K\"{a}hler class, by Lemma
\ref{lemma-slope-stability-does-not-depend-on-kahler-class}. 
We will abuse notation and
say that a sheaf on $X_{t_i}$ is $\omega_{t_i}$-slope-stable,
if it is slope-stable with respect to some, hence any K\"{a}hler class.


\begin{prop}
\label{prop-deformation-of-azumaya-algebras-along-twistor-paths}
\begin{enumerate}
\item
\label{prop-item-untwisted}
Let $F$ be an $\omega_0$-slope-stable (possibly twisted) reflexive sheaf.
Assume that $\kappa_2(F)$ 
remains of Hodge type $(2,2)$ along $\gamma$
and the first Chern class of every direct summand of $\SheafEnd(F)$ vanishes.   
Then $F$ deforms along $\gamma$,
in the sense of definition \nolinebreak
\ref{def-can-be-deformed-along-a-twistor-path}. 
\item
\label{prop-item-twisted}
Let $A$ be an $\omega_0$-slope-stable reflexive sheaf of Azumaya algebras of rank $r$
(Definition \ref{def-slope-stability}). 
Assume
that $c_2(A)$ 
remains of Hodge type $(2,2)$ along $\gamma$
and the first Chern class of every direct summand of $A$ vanishes. 
Then $A$ deforms along $\gamma$, as a reflexive sheaf of Azumaya algebras, 
in the sense of definition
\ref{def-can-be-deformed-along-a-twistor-path}. 
\end{enumerate}
\end{prop}

\begin{proof} 
(\ref{prop-item-twisted})
The following argument is similar to the proof of  \cite[Theorem 10.8]{kaledin-verbitski-book}.
The proof is by induction on the number $N$ of twistor lines in $C$.
$A$ deforms along the first twistor line, by  Corollary 
\ref{cor-main-result-on-projectively-hyperholomorphic-reflexive-sheaves} and Remark \ref{rem-caution-not-flat}.

Assume that $A$ deforms, as an $\omega_t$-slope-stable sheaf of
Azumaya algebras, along the first $i$ twistor lines, and $i<N$. Then
$A_{t_i}$ is slope-polystable with respect to $\omega_{t_i^-}$ and hence also
with respect to $\omega_{t_i^+}$, by Lemma 
\ref{lemma-slope-stability-does-not-depend-on-kahler-class}.
Hence, $A_{t_i}$ is $\omega_{t_i^+}$ polystable-hyperholomorphic, by
Theorem \ref{thm-Verbitskys-inequality} and Lemma \ref{lemma-Mon-invariant-classes-are-of-Hodge-type}. The structure
of an Azumaya algebra deforms along the 
$i+1$ twistor line, by Lemma 
\ref{lemma-extension-of-a-slope-zero-subsheaf-to-a-hyperholomorphic-subsheaf}
part \ref{lemma-item-algebraic-structures}. 

The $\omega_t$-slope-stability  of $A_t$ is proven by induction as well.
The underlying rank $r^2$ coherent sheaf $A$ is $\omega_0$-slope-polystable, by Proposition \ref{prop-polystability-of-a-very-twisted-sheaf}, since $A\cong\SheafEnd(F)$ for some $\omega_0$-slope-stable reflexive twisted sheaf $F$ (see Section \ref{sec-azumaya}).
The stability for $t$ in the first twistor line follows from Lemma
\ref{lemma-if-stable-for-some-t-then-stable-for-all-t}. 
Stability of $A_{t_1}$
for $\omega_{t_1^+}$ follows from that for $\omega_{t_1^-}$ and Lemma 
\ref{lemma-slope-stability-does-not-depend-on-kahler-class}.
The proof of the induction step is similar.

Part (\ref{prop-item-untwisted})
follows from part (\ref{prop-item-twisted}), since 
$\SheafEnd(F)$ is an $\omega_0$-slope-stable sheaf of Azumaya algebras,
the underlying sheaf $\SheafEnd(F)$ (forgetting the algebra structure) is
$\omega_0$-slope-polystable by Proposition \ref{prop-polystability-of-a-very-twisted-sheaf}, 
and $c_2(\SheafEnd(F))$ is a scalar multiple of $\kappa_2(F)$, by Lemma \ref{lemma-kappa-2-is-propotional-to-c-2-of-End}. 
\end{proof}

\begin{rem}
\label{rem-verbitsky-results-ok-for-products}
With the exception of Theorem \ref{thm-path-connectedness}, Verbitsky proves the 
results mentioned
above for hyperk\"{a}hler varieties, without assuming the condition $h^{2,0}=1$
(the irreducibility condition). In particular, all the definitions and results in this section
hold for $X\times X$, where $X$ is an irreducible holomorphic symplectic manifold,
provided the twistor deformations of $X\times X$ we consider are only 
fiber-square  $\X\times_{\PP^1_\omega} \X$ of twistor deformations of $X,$
associated to a K\"{a}hler class $\omega$ on $X$.
Corollary \ref{cor-main-result-on-projectively-hyperholomorphic-reflexive-sheaves} and 
Proposition \ref{prop-deformation-of-azumaya-algebras-along-twistor-paths} 
will be applied in this form for $X$ replaced by $X\times X$ in the proofs of
Theorem \ref{thm-deformations-of-E-along-twistor-paths}.
\end{rem}
%
\section{Stable hyperholomorphic sheaves of rank $2n-2$ on all 
manifolds of $K3^{[n]}$-type}
\label{sec-deformability-theorem}

\begin{defi}
Let $X$ be a complex manifold and $E$ a torsion free  $\theta$-twisted coherent sheaf
on $X$. $E$ is said to be {\em very twisted}, if the rank of $E$ is equal to the order of the 
class of $\theta$ in $H^2_{an}(X,\StructureSheaf{X}^*)$.
\end{defi}

A very-twisted reflexive sheaf does not have any non-trivial subsheaves of lower rank, by Remark \ref{rem-order-divides-rank}, 
so it is trivially slope-stable with respect to every K\"{a}hler class. The following Lemma thus applies.

\begin{lem}
\label{lemma-uniform-slope}
Let $E$ be a reflexive possibly twisted sheaf over a compact K\"{a}hler manifold $X$. 
Assume that $E$ is $\omega$-slope-stable for all K\"{a}hler classes $\omega$ in some open subset $U$ of the
K\"{a}hler cone of $X$. Then the first Chern class of every direct summand of
$\SheafEnd(E)$ vanishes.
\end{lem}

\begin{proof}
$\SheafEnd(E)$ is $\omega$-slope-polystable with respect to every K\"{a}hler class $\omega$ in $U$, by Proposition
\ref{prop-polystability-of-a-very-twisted-sheaf}. Set $n:=\dim_\ComplexNumbers(X)$. 
Hence, every direct summand of $\SheafEnd(E)$ has slope zero with respect to every K\"{a}hler class in $U$.
The image of $U$ under the polynomial map
$\omega\mapsto \omega^{n-1}$ is an open subset of $H^{n-1,n-1}(X,\RealNumbers)$, since its differential is invertible at every point, by Hard Lefschetz. 
Hence, the first Chern class of every direct summand vanishes.
\end{proof}

\begin{rem}
Slope stability of a torsion free sheaf is known to be an open condition on the K\"{a}hler class in many cases.
See \cite[Sec. 5.1]{lubke-teleman} for locally free sheaves over compact K\"{a}hler manifolds, \cite[Sec. 4.C]{huybrechts-lehn-book} for torsion free sheaves over projective surfaces, and \cite{greb-ross-toma} for more recent results for higher dimensional projective varieties.
\end{rem}

In Section \ref{sec-a-very-twisted-ext-1} 
we construct a very twisted version of the sheaf $E$ in
Proposition \ref{prop-V}.
In Section \ref{sec-proof-of-deformability}
we prove the deformability Theorem 
\ref{thm-main-introduction}.

%
\subsection{A very twisted 
$\SheafExt^1_{\pi_{13}}(\pi_{12}^*\E,\pi_{23}^*\E)$}
\label{sec-a-very-twisted-ext-1}
We construct a very twisted reflexive sheaf 
$\SheafExt^1_{\pi_{13}}(\pi_{12}^*\E,\pi_{23}^*\E)$,
over the self-product of a suitable choice of a moduli space $\M$
(Theorem \ref{thm-E-is-slope-stable}). 

Let $\M_H(v)$ be a smooth and projective moduli space of
$H$-stable sheaves on a projective $K3$ surface $S$. 
Set $r:=(v,v)$. Assume, that $(v,v)\geq 2$. 
Let $\mu_r$ be the group of $r$-th roots of unity.
Let $\exp:H^2(\M_H(v),\frac{2\pi\sqrt{-1}}{r}\Integers)\rightarrow H^2(\M_H(v),\mu_r)$ be the homomorphism in Equation 
(\ref{eq-tilde-theta-via-c-1}).

\begin{lem}
\label{lemma-exp-bar-w-is-tilde-theta}
\begin{enumerate}
\item
\label{lemma-item-existence-and-uniqueness-of-w-bar}
There exists a unique $r(v^\perp)$ coset $\bar{w}$ in $v^\perp$
of classes $w$, such that $w-v$ belongs to $rK_{top}S$. 
\item
\label{lemma-item-pair-tilde-theta-is-monodromy-invariant}
Define a class in $H^2(\M_H(v),\mu_r)$, by 
\begin{equation}
\label{eq-tilde-theta-is-exp-bar-w}
\tilde{\theta}:=\exp(-2\pi\sqrt{-1}\bar{w}/r),
\end{equation}
where we identify $v^\perp$ with $H^2(\M_H(v),\Integers)$ via
Mukai's isometry (\ref{eq-Mukai-isomorphism}).
Then the pair $\{\tilde{\theta},\tilde{\theta}^{-1}\}$ is monodromy invariant.
\end{enumerate}
\end{lem}

\begin{proof}
\ref{lemma-item-existence-and-uniqueness-of-w-bar})
Uniqueness is clear. When $v$ is the class of the ideal sheaf 
of a length $n$ subscheme, with Mukai vector $(1,0,1-n)$,
choose $w=(1,0,n-1)$. The existence of such a class follows, for an arbitrary
primitive class $v$ with $(v,v)=2n-2$, since any two such classes
belong to the same $O(K_{top}S)$-orbit. 

\ref{lemma-item-pair-tilde-theta-is-monodromy-invariant})
The class $\tilde{\theta}$ is determined by the primitive isometric 
lattice embedding 
$H^2(\M_H(v),\Integers)\cong v^\perp\subset K_{top}S$
and the choice of a generator $v$ of the line orthogonal to the image of
$H^2(\M_H(v),\Integers)$. Any monodromy operator of 
$H^2(\M_H(v),\Integers)$ can be extended to an isometry of $K_{top}S$,
which necessarily maps $v$ to $v$ or $-v$, by
Theorem \ref{thm-weight-2-monodromy}.
\end{proof}

Let $\tilde{\theta}$ be the class in Equation (\ref{eq-tilde-theta-is-exp-bar-w}).
Denote by $\theta$ the image of $\tilde{\theta}$
in $H^2(\M_H(v),\StructureSheaf{}^*)$, via the sheaf inclusion
$\iota:\mu_r\hookrightarrow \StructureSheaf{}^*$. 
Let $\beta:B\rightarrow \M_H(v)\times \M_H(v)$  be the 
blow-up of the diagonal and $\PP{V}$ the projective bundle over $B$
associated to the twisted locally free sheaf (\ref{eq-vector-bundle-over-B}).

\begin{lem}
\label{lem-order-of-twisting-class-of-universal-sheaf}
\begin{enumerate}
\item
\label{lemma-item-obstruction-class-of-lifting-PP-V-to-SL-r}
The class $\tilde{\theta}(\PP{V})\in H^2(B,\mu_r)$, defined in
(\ref{eq-connecting-homomorphism-theta-tilde}),
satisfies
\begin{equation}
\label{eq-tilde-theta-PP-V}
\tilde{\theta}(\PP{V}) \ \ \ = \ \ \ 
\beta^*\left((\pi_1^*\tilde{\theta})^{-1}\pi_2^*\tilde{\theta}\right).
\end{equation} 
\item
\label{lemma-item-order-of-theta}
The order of the class $\theta$ in $H^2_{an}(\M_H(v),\StructureSheaf{}^*)$
is given by:
\[
\gcd\{(v,x) \ : \ x\in K_{top}S \ \mbox{and} \ 
c_1(x) \ \mbox{is of type} \ (1,1)\}.
\]
\end{enumerate}
\end{lem}

\begin{proof}
\ref{lemma-item-obstruction-class-of-lifting-PP-V-to-SL-r})
Assume first, that $v$ is the class of the ideal sheaf of a length $n$ 
subscheme. Then $V$ is a vector bundle, which restricts to the
exceptional divisor $D$ as a vector bundle with trivial determinant
(Proposition \ref{prop-V}). Thus,
$c_1(V)=\beta^*\beta_*c_1(V)=
\beta^*c_1(\F)$, where $\F$ is the object given in Equation
(\ref{eq-object-F}). Now, 
$c_1(\F)=-\pi_1^*c_1(e_v)+\pi_2^*c_1(e_v)$,
by Lemma \ref{lem-c1-of-F}. 
When $\E$ is the universal ideal sheaf over $S\times S^{[n]}$,
then $c_1(e_v)=c_1(e_w)$, where $v$ has Mukai vector $(1,0,1-n)$, and 
that of $w$ is $(1,0,n-1)$, by \cite[Lemma 5.9]{markman-integral-constraints}.
The coset $\bar{w}$ in equation (\ref{eq-tilde-theta-is-exp-bar-w}) is
$w+(2n-2)K_{top}S$, since $w-v=(2n-2)(0,0,1)$.
The equality (\ref{eq-tilde-theta-PP-V})
follows from equation (\ref{eq-tilde-theta-via-c-1}). 

The general case of equation (\ref{eq-tilde-theta-PP-V}) follows,
by deformation of the classes on both sides, via a 
deformation to the Hilbert scheme case, as in Lemma 
\ref{lem-lifting-to-deformations-of-pairs-ok-for-moduli-spaces}. 

\ref{lemma-item-order-of-theta})
Set $\M:=\M_H(v)$. Consider the short exact sequence
\begin{equation}
\label{eq-short-exact-seq-of-r-th-power}
0\rightarrow \mu_r \RightArrowOf{\iota} \StructureSheaf{}^* 
\RightArrowOf{(\bullet)^r} \StructureSheaf{}^* \rightarrow 0.
\end{equation}
The connecting homomorphism 
$H^1(\M,\StructureSheaf{}^*)\rightarrow H^2(\M,\mu_r)$ sends 
the class of a line bundle $L$ to $\exp(2\pi\sqrt{-1}c_1(L)/r)$. 
Let $d$ be a positive integer dividing $(v,v)$. 
Then $\iota(d\tilde{\theta})=1$, if and only if 
$d\tilde{\theta}=\exp(-2\pi\sqrt{-1}\ell/r)$, 
for some $\ell\in H^{1,1}(\M,\Integers)$. 
Identify $H^2(\M,\Integers)$ with $v^\perp$, via Mukai's Hodge-isometry
(\ref{eq-Mukai-isomorphism}). 
Set $\bar{\ell}:=\ell+rv^\perp$. 
It suffices to prove that the following are equivalent.
\begin{enumerate}
\item
\label{statement-ell-exists}
There exists $\ell\in v^\perp$, with $c_1(\ell)$ of type $(1,1)$, 
such that $\bar{\ell}=d\bar{w}$ in $v^\perp/rv^\perp$, where
$\bar{w}$ is the coset in Lemma \ref{lemma-exp-bar-w-is-tilde-theta}. 
\item
\label{statement-x-exists}
$d=(v,x)$, for some $x\in K_{top}S$, with $c_1(x)$ of type $(1,1)$.
\end{enumerate}

\ref{statement-ell-exists}$\Rightarrow$ \ref{statement-x-exists}:
The $(1,1)$ class $x:=\frac{dv-\ell}{r}$ is integral, 
by the assumption on $\ell$, and satisfies 
$(x,v)=\frac{d(v,v)}{r}=d$.

\ref{statement-x-exists}$\Rightarrow$\ref{statement-ell-exists}:
Set $\ell:=dv-(v,v)x$. Then $(\ell,v)=0$ and 
$\ell-dv=-rx$ belongs to
$rK_{top}S$. Thus, $\bar{\ell}=d\bar{w}$ in $v^\perp/rv^\perp$,
by Lemma \ref{lemma-exp-bar-w-is-tilde-theta} part 
\ref{lemma-item-existence-and-uniqueness-of-w-bar}.
\end{proof}

Set $r:=2n-2$, $n\geq 2$. Let $S$ be a projective $K3$ surface with
a cyclic Picard group generated by an ample line bundle $H$ with 
$c_1(H)^2=2r^2+r$. Let $v\in K_{top}S$ be the rank $r$ class
with $c_1(v)=c_1(H)$, and $\chi(v)=2r$. Its
Mukai vector $ch(v)\sqrt{td_S}$ is $(r,H,r)$. 
Then $(v,v)=r$ and $(v,x)\equiv 0$, (modulo $r$), 
for every class $x\in K_{top}S$ with $c_1(x)$ of type $(1,1)$. 
The moduli space $\M_H(v)$ is smooth and projective (see Section 
\ref{sec-construction-of-the-classes-kappa-i-X}).
Let $E$ be the rank $r$
$(\pi_1^*[\theta]^{-1}\pi_2^*[\theta])$-twisted sheaf
$\SheafExt^1_{\pi_{13}}(\pi_{12}^*\E,\pi_{23}^*\E)$,
over $\M_H(v)\times \M_H(v)$. 
$E$ is reflexive, by Proposition \ref{prop-V}.

\begin{thm}
\label{thm-E-is-slope-stable}
The $(\pi_1^*[\theta]^{-1}\pi_2^*[\theta])$-twisted sheaf $E$ is 
$\omega$-slope-stable (Definition \ref{def-slope-stability}) and the untwisted sheaf 
$\SheafEnd(E)$ is $\omega$-polystable-hyperholomorphic
(Definition \ref{def-hyperholomorphic-sheaf}), 
with respect to every K\"{a}hler class $\omega$ on $\M_H(v)\times \M_H(v)$.  
\end{thm}

\begin{proof}
The class $\theta$ has order $r$, by Lemma
\ref{lem-order-of-twisting-class-of-universal-sheaf}. 
It follows that $E$ does not have any non-zero 
twisted subsheaves of rank $< r$ (see Remark \ref{rem-order-divides-rank}). 
The $\omega$-slope-polystability of $\SheafEnd(E)$ follows from  
Proposition \ref{prop-polystability-of-a-very-twisted-sheaf} for all K\"{a}hler classes $\omega$.
Recall that $c_2(\SheafEnd(E))$ is a scalar multiple of $\kappa_2(E)$, by Lemma \ref{lemma-kappa-2-is-propotional-to-c-2-of-End}.
The class $\kappa_2(E)$ is monodromy invariant,
by Proposition \ref{prop-kappa-F-is-mon-invariant}. 
The first Chern classes of all direct summands of $\SheafEnd(E)$ vanish, by Lemma \ref{lemma-uniform-slope}.
Consequently, $\SheafEnd(E)$ is $\omega$-polystable-hyperholomorphic,
by Theorem \ref{thm-Verbitskys-inequality}, Lemma \ref{lemma-Mon-invariant-classes-are-of-Hodge-type}, and  
Remark \ref{rem-verbitsky-results-ok-for-products}.
\end{proof}

\subsection{Proof of the deformability  Theorem 
\ref{thm-main-introduction}}
\label{sec-proof-of-deformability}
Let $E$ be the very twisted sheaf of Theorem \ref{thm-E-is-slope-stable} over $\M_H(v)\times \M_H(v)$.

\begin{thm}
\label{thm-deformations-of-E-along-twistor-paths}
Let 
$X$ be an irreducible holomorphic symplectic manifold of $K3^{[n]}$-type. 
Then there exists a parametrized twistor path connecting $\M_H(v)$ and $X$,
along which $E$ can be deformed 
(in the sense of Definition \ref{def-can-be-deformed-along-a-twistor-path}).
\end{thm}

\begin{proof}
The class $\kappa_2(E)$ is $Mon(\M_H(v))$-invariant, 
by Proposition \ref{prop-kappa-F-is-mon-invariant}.
Let $\omega$ be a K\"{a}hler class on $\M_H(v)$ and set $\tilde{\omega}:=\pi_1^*\omega+\pi_2^*\omega$, where
$\pi_i$ is the projection from $\M_H(v)\times \M_H(v)$ onto the $i$-th factor. 
Then $\SheafEnd(E)$ is $\tilde{\omega}$-slope-polystable, by Proposition \ref{prop-polystability-of-a-very-twisted-sheaf}.
The sheaf 
$E$ is projectively $\omega$-stable-hyperholomorphic, by 
Corollary 
\ref{cor-main-result-on-projectively-hyperholomorphic-reflexive-sheaves}, 
Lemma \ref{lemma-Mon-invariant-classes-are-of-Hodge-type},
and Remark \ref{rem-verbitsky-results-ok-for-products}.
We may choose $\omega$, so that the hyperplane $\omega^\perp$ intersects trivially the 
lattice $H^{1,1}(\M_H(v),\Integers)$. Then 
$\Pic(X_{t_1})$ is trivial, for a generic $t_1\in \PP^1_\omega$, by \cite[paragraph 1.17
page 76]{huybrects-basic-results}. 
There exists a generic parametrized twistor path from $X_{t_1}$ to $X$, 
by Theorem \ref{thm-path-connectedness}. 
We get a generic parametrized twistor path from $\M_H(v)$ to $X$.
We conclude that $E$ deforms 
along the twistor path $\gamma$, by Proposition 
\ref{prop-deformation-of-azumaya-algebras-along-twistor-paths}. 
\end{proof}

\begin{proof}[Proof of Theorem \ref{thm-main-introduction}]
It remains to prove the equality $\kappa_i(F)=\pm\kappa_i(X\times X)$ for the sheaf $F$ obtained on $X\times X$ as a deformation of the sheaf $E$ via Theorem
\ref{thm-deformations-of-E-along-twistor-paths}, for $2\leq i \leq 2n-1$. 
The pair $\{\kappa(E),\kappa(E^*)\}$ associated to the sheaf $E$ in Theorem
\ref{thm-deformations-of-E-along-twistor-paths} is a parallel transport of the pair of $\kappa$-classes associated to the sheaf in Equation (\ref{eq-E}), by Lemma \ref{lem-lifting-to-deformations-of-pairs-ok-for-moduli-spaces}. 
Let $\Pi:\X\times_C\X\rightarrow C$ be the twistor family over the twistor path $C$ and let 
$\A$ be the Azumaya algebra over $\X\times_C\X$ extending $\SheafEnd(E)$  in the proof of Theorem
\ref{thm-deformations-of-E-along-twistor-paths}.
The equality $\kappa_i(F)=\pm\kappa_i(X\times X)$ would be clear, for all $i$, 
had we known the flatness of $\A$ over $C$. 
We do know that the singular locus $\Z$ of $\A$ is equidimensional\footnote{In fact, $\A$ is locally free away from 
the image of the diagonal embedding of $\X$ in its fiber square 
$\X\times_C\X$, by Proposition \ref{prop-V} and the fact that the singular locus is trianalytic. } 
over $C$, by Theorem \ref{thm-deformations-of-E-along-twistor-paths} and Definition \ref{def-can-be-deformed-along-a-twistor-path}. 
Let $U_t\subset X_t\times X_t$ be the complement of the intersection $Z_t$ of $\Z$ with the fiber over a point $t$ in $C$. 
We have  $\dim(Z_t)=2n$, since  $E$ is locally free away from the diagonal, 
by Proposition \ref{prop-V}. 
Recall that the Azumaya algebra $\A_t$ over $X_t\times X_t$ is the unique reflexive extension of the restriction of $\A$ to $U_t$, by Construction \ref{construction-family-of-projectively-hyperholomorphic-sheaves}. It suffices to show that 
$\kappa_i(\A_t)$ is equal to the restriction of $\kappa_i(\A)$ to $X_t\times X_t$,  for $2\leq i \leq 2n-1$, since it would then follow that 
the characteristic classes $\kappa_i(\A_t)$ form flat sections of the local system 
$R^{2i}\Pi_*\RationalNumbers$ over $C$, for $i$ in that range. 

The restrictions of $\kappa_i(\A)$ and $\kappa_i(\A_t)$ to $H^{2i}(U_t,\RationalNumbers)$ are equal, since both are equal to the $\kappa_i$ class of the restriction of $\A$ to $U_t$. 
The restriction homomorphism $H^k(X_t\times X_t,\Integers)\rightarrow H^k(U_t,\Integers)$ is an isomorphism, for $k\leq 4n-2$,
by Lefschetz Duality $H^k(U_t,\Integers)\cong H_{8n-k}(X_t\times X_t,Z_t,\Integers)$ and the vanishing of $H_{8n-k}(Z_t,\Integers)$ for $k<4n$. Hence, the restriction of $\kappa_i(\A)$ to $X_t\times X_t$ is equal to $\kappa_i(\A_t)$, for $2\leq i \leq 2n-1$. 
\end{proof}

\section{Proof of Lemma \ref{lem-BB-form-is-motivic}}
\label{sec-proof-of-relation-between-kappa-2-c-2-and-BB}
It suffices to prove the Lemma for every smooth and compact moduli space
$\M:=\M_H(v)$, for all $(v,v)\geq 2$. 
Let 
\begin{eqnarray*}
u \ : \ K_{top}S & \longrightarrow & H^*(\M,\RationalNumbers)
\\
u(x) & := & ch(e_x)\cdot \exp\left(\frac{-c_1(e_v)}{(v,v)}
\right),
\end{eqnarray*}
where $e_x$ is given in (\ref{eq-e-x}), 
and $u_{2i}:K_{top}S\rightarrow H^{2i}(\M,\RationalNumbers)$
the composition of $u$ with the projection on the degree $2i$-summand.
Note that $u(v)=\kappa(e_v)$, $u_0(x)=(v,x)$, 
\[
u_2(x) \  \ \ = \ \ \ 
c_1(e_x)-\frac{(v,x)}{(v,v)}c_1(e_v),
\]
$u_2(v)=0$, and 
$u_2$ restricts to $v^\perp$ as the standard Mukai isomorphism of Equation (\ref{eq-Mukai-isomorphism})
\[
(u_2\restricted{)}{v^\perp} \ : \ v^\perp \ \ \ 
\LongIsomRightArrow H^2(\M,\Integers). 
\]
Moreover, $u$ is $O^+(K_{top}S)_v$ equivariant,
$
\monrep_g(u(g^{-1}(x))=u(x).
$
Indeed, 
\begin{eqnarray*}
\monrep_g(u(g^{-1}(x))&=& 
\monrep_g\left(ch(e_{g^{-1}(x)})\exp(-c_1(v)/(v,v))\right)\stackrel{{\rm Eq.} \ (\ref{eq-mon-equivariance-of-e})}{=}\\
& = & ch(e_x)\exp\left(c_1(\ell_g)-\monrep_g(c_1(e_v))/(v,v)\right) \stackrel{{\rm Eq.} \ (\ref{eq-c-1-ell-g})}{=} u(x).
\end{eqnarray*}

The Mukai pairing is a class in $\Sym^2(K_{top}S)^*$. It determines an isomorphism
$K_{top}S\rightarrow K_{top}S^*$, being unimodular. The inverse of the latter isomorphism corresponds to a class 
$\tilde{q}$ in $\Sym^2K_{top}S$. 
The following equality is a special case of 
\cite[Eq. (4.8)]{markman-monodromy-I}:
\begin{equation}
\label{eq-c-2-TM-is-image-of-Mukai-pairing}
c_2(T\M) \ \ \ = \ \ \ (u_2\cup u_2 - 2u_4\cup u_0)(\tilde{q}),
\end{equation}
where $(u_2\cup u_2 - 2u_4\cup u_0)$ is the homomorphism from 
$K_{top}S\otimes K_{top}S$ to $H^4(\M,\RationalNumbers)$. 

The orthogonal decomposition 
$(K_{top}S)_\RationalNumbers=\RationalNumbers v+(v^\perp)_\RationalNumbers$
induces the decomposition
$
\tilde{q}  = \frac{v\otimes v}{(v,v)} + q^{-1},
$
where we identified $v^\perp$ with $H^2(\M,\Integers)$, via $u_2$. 
Equation (\ref{eq-relation-between-kappa-2-c-2-and-BB})
follows from (\ref{eq-c-2-TM-is-image-of-Mukai-pairing})
and the following equations
\begin{eqnarray}
\label{eq-u-4-cup-u-0-kills-q-inverse}
(u_4\cup u_0)(q^{-1}) & = & 0,
\\
\label{eq-tautology}
(u_2\cup u_2)(q^{-1}) &= & q^{-1},
\\
\label{eq-u-2-cup-u-2-of-v-times-v}
(u_2\cup u_2)(v\otimes v) & = & 0,
\\
\label{eq-u-4-cup-u-0-of-v-times-v}
(u_4\cup u_0)\left(\frac{v\otimes v}{(v,v)}\right) & = & u_4(v) \ = \ 
\kappa_2(X).
\end{eqnarray}

Proof of Equation (\ref{eq-u-4-cup-u-0-kills-q-inverse}):
$u_4\cup u_0$ is $O^+(K_{top}S)_v$-equivariant, and thus sends the 
$O^+(K_{top}S)_v$-invariant class 
$q^{-1}$ in $(v^\perp\otimes v^\perp)_\RationalNumbers$ to an 
$O^+(K_{top}S)_v$-invariant class in $u_4(v^\perp)_\RationalNumbers$. 
But the image $u_4(v^\perp)$ either vanishes, or is an irreducible 
$O(K_{top}S)_v$-module isomorphic to $v^\perp$. Thus, any
invariant class in $u_4(v^\perp)$ vanishes.

Equations (\ref{eq-tautology}) and (\ref{eq-u-4-cup-u-0-of-v-times-v}) 
are clear and Equation
(\ref{eq-u-2-cup-u-2-of-v-times-v}) follows from the vanishing of
$u_2(v)$, observed above.

It remains to calculate the dimension of 
$\mbox{span}\{q^{-1},c_2(TX),\kappa_2(X)\}$. 
$\Sym^2H^2(S^{[n]},\RationalNumbers)$ is the direct sum of the line spanned by $q^{-1}$ and the subspace spanned by squares of isotropic vectors, and the latter is an irreducible representation of any finite index subgroup of the orthogonal group 
\cite[Prop. 2.14]{looijenga-lunts}, hence also of 
$Mon(S^{[n]})$. Hence, the monodromy invariant subspace of $\Sym^2H^2(S^{[n]},\RationalNumbers)$ is one dimensional.
The homomorphism $\Sym^2H^2(S^{[n]},\RationalNumbers)\rightarrow 
H^4(S^{[n]},\RationalNumbers)$ is known to be injective
\cite{verbitsky}. 
When $n=2$, the homomorphism is surjective, by G\"{o}ttsche's
formula for the Betti numbers \cite{gottsche-formula}. 
When $n=3$, the co-kernel of the homomorphism is
an irreducible $23$-dimensional representation of $Mon(S^{[3]})$
\cite{markman-monodromy-I}.
Thus, the monodromy invariant subspace of $H^4(X,\RationalNumbers)$
is one dimensional, and is spanned by each of 
the three classes, 
for $X$ of $K3^{[n]}$-type, $n\leq 3$. 

Assume that $n\geq 4$. Then the monodromy invariant subspace of the quotient space \\
$H^4(S^{[n]},\RationalNumbers)/\Sym^2H^2(S^{[n]},\RationalNumbers)$
is one-dimensional and is spanned by the image of each of $\kappa_2(X)$
and $c_2(TX)$ \cite[Lemma 4.9]{markman-monodromy-I}.
\EndProof

\section{Appendix: Polystability of $\SheafEnd(E)$ for a slope-stable twisted sheaf $E$}
\label{sec-semistability}
We prove Proposition \ref{prop-polystability-of-a-very-twisted-sheaf} in this section. Given a coherent sheaf $F$ over a complex manifold we denote by $F_{fr}$ the quotient of $F$ by its torsion subsheaf. The following is well known (see \cite[Sec. 3.5]{kaledin-verbitski-book}).

\begin{lem}
\label{lemma-polystability-and-tensor-products}
\cite{bando-siu}
Let $E$ and $F$ be reflexive coherent sheaves on a compact K\"{a}hler manifold $X$ and $\omega$ a K\"{a}hler form. 
If $E$ and $F$ are $\omega$-slope-polystable, then so is the reflexive hull of $(E\otimes F)_{fr}$. 
If $E$ and $F$ are $\omega$-slope-semistable, then so is $(E\otimes F)_{fr}$. 
\end{lem}

\begin{proof}
According to Bando and Siu, a reflexive coherent sheaf is $\omega$-slope-polystable if and only if it admits an admissible Hermite-Einstein metric \cite[Theorem 3]{bando-siu}. If $E$ and $F$ are $\omega$-slope-polystable, the  metric induced on the reflexive hull $(E\otimes F)_{fr}^{**}$ from  admissible Hermite-Einstein metrics of the factors is again admissible Hermite-Einstein and so $(E\otimes F)_{fr}^{**}$
is $\omega$-slope-polystable as well. 

The sheaf $(E\otimes F)_{fr}$ is $\omega$-slope-polystable (or semistable), if and only if its reflexive hull is. 
Now a sheaf is $\omega$-slope-semistable if and only if it admits a filtration whose graded summands are 
$\omega$-slope-polystable of the same slope. Such filtrations on $E$ and $F$ induce a filtration on 
$(E\otimes F)_{fr}^{**}$ by $\omega$-slope-polystable sheaves of the same slope, by the previous paragraph.
Hence, semistability of $E$ and $F$ implies that of $(E\otimes F)_{fr}^{**}$.
\end{proof}

\begin{defi}
Let $X$ be a complex manifold and $E$ a torsion free  $\theta$-twisted coherent sheaf on $X$. 
A subsheaf $A\subset \SheafEnd(E)$ is said to be {\em nilpotent}, if there exists a  filtration 
\[
0=V_0 \subset V_1 \subset V_2 \subset \cdots \subset V_k=E
\] 
of $E$ by subsheaves $V_i$ of strictly increasing ranks, such that the image of the natural homomorphism 
$A\otimes V_i\rightarrow E$ is contained in $V_{i-1}$, for $1\leq i\leq k$. 
\end{defi}

We may and will choose the  $V_i$ in the above filtration to be saturated subsheaves of $E$. 

\begin{rem}
\label{rem-Engel-theorem}
A subsheaf $A\subset \SheafEnd(E)$ is said to be a {\em subsheaf  of Lie subalgebras}, if the commutators $a_1a_2-a_2a_1$ of local sections of $A$ belong to $A$. 
Any subsheaf $A\subset \SheafEnd(E)$ of Lie subalgebras, whose local sections are nilpotent, is a nilpotent subsheaf, by Engel's Theorem \cite[Corollary in Sec. I.3.3]{humphreys}.
\end{rem}

\begin{lem}
\label{lem-non-semistable-has-nilpotent-subsheaf-of-positive-slope}
Let $(X,\omega)$ be a compact K\"{a}hler manifold and $E$ a reflexive 
$\theta$-twisted coherent sheaf, for some $\theta\in H^2_{an}(X,\StructureSheaf{}^*)$.
If $\SheafEnd(E)$ is not $\omega$-slope-semistable then there exists a non-zero nilpotent $\omega$-slope-stable saturated subsheaf $A$ of $\SheafEnd(E)$ of positive $\omega$-slope, which is equal to the maximal slope among all (not necessarily nilpotent) subsheaves of $\SheafEnd(E)$.
\end{lem}

\begin{proof}
Assume that $\SheafEnd(E)$ is not semistable and let $A$ be an
$\omega$-slope-stable destabilizing subsheaf of $\SheafEnd(E)$ of maximal slope. The existence of $A$ follows by 
\cite[Theorem 1.6.7]{huybrechts-lehn-book}. The latter relies on the argument in the proof of 
\cite[Lemma 1.3.5]{huybrechts-lehn-book}, with Gieseker stability replaced by slope stability, an argument which goes through for coherent sheaves on compact K\"{a}hler manifolds.
We may assume that $A$ is a saturated subsheaf, since the slope of its saturation is greater than or equal to that of $A$. 
Then $(A\otimes A)_{fr}$ is $\omega$-slope-polystable of slope $2\mu_\omega(A)$.
The image of $A\otimes A$ in $\SheafEnd(E)$ must be zero, since otherwise 
the slope of the image is at least $2\mu_\omega(A)$, contradicting the assumption
that the slope of $A$ is maximal. We conclude that $A$ is a subsheaf of nilpotent subalgebras, hence a
nilpotent subsheaf, by Remark \ref{rem-Engel-theorem}. 
\end{proof}

\begin{lem}
\label{lemma-non-polystable-has-nilpotent-subsheaf-of-zero-slope}
Let $(X,\omega)$ be a compact K\"{a}hler manifold and $E$ a reflexive  
$\theta$-twisted  coherent sheaf, for some $\theta\in H^2_{an}(X,\StructureSheaf{}^*)$.
If $\SheafEnd(E)$ is not $\omega$-slope-polystable then there exists a non-zero nilpotent $\omega$-slope-stable saturated subsheaf $A$ of $\SheafEnd(E)$ of non-negative $\omega$-slope.
\end{lem}

\begin{proof}
We may assume that $\SheafEnd(E)$ is $\omega$-slope-semistable, as otherwise the statement follows from 
Lemma \ref{lem-non-semistable-has-nilpotent-subsheaf-of-positive-slope}. 
Let $F\subset \SheafEnd(E)$ be the maximal polystable subsheaf of $\omega$-slope zero
\cite[Lemma 1.5.5]{huybrechts-lehn-book}.
Then $F$ is reflexive, and is hence locally free away from a closed analytic
subvariety $Z$ of codimension $\geq 3$ in $X$. 
Let $F^\perp\subset \SheafEnd(E)$ be the subsheaf orthogonal to $F$ 
with respect to the trace-pairing on $\SheafEnd(E)$. We may assume that $\SheafEnd(E)$ is not $\omega$-slope-polystable.
Then $F^\perp$ does not vanish.
Set $A:=F\cap F^\perp$. 

The multiplication homomorphism
\[
m:\SheafEnd(E)\otimes \SheafEnd(E)\rightarrow \SheafEnd(E)
\]
maps $F\otimes F$ onto a subsheaf of slope $0$, since  $(F\otimes F)_{fr}$ is $\omega$-slope polystable 
and $\SheafEnd(E)$ is $\omega$-slope semistable and both have slope $0$.
We conclude that the image is slope-polystable, and is hence contained in $F$.
Consequently, $F$ is a sheaf of unital associative subalgebras of $\SheafEnd(E)$.

We show next that $A$ is a subsheaf of associative subalgebras, and in particular of Lie subalgebras. 
Let $a_1$ and $a_2$ be local sections of $A$ and $f$ a local section of $F$. 
Then 
$tr((a_1a_2)f)=tr(a_1(a_2f))=0$, since $a_1$ is a local section of $F^\perp$ and
$a_2f$ is a local section of  $F$, by the previous paragraph.
We conclude that $a_1a_2$ is a local section of $F^\perp$. Now $a_1a_2$ is a local section of $F$ as well, by the previous paragraph, and so of $A$.

Let $a$ be a local section of $A$.
Then $a^n$ is a section of $F$, for all $n\geq 0$.
Thus, $tr(a^k)=tr(a^{k-1}a)=0$, for all $k>0$. 
It follows that $a$ is nilpotent. Hence, the sheaf $A$ is a nilpotent subsheaf of 
$\SheafEnd(E)$, by Remark \ref{rem-Engel-theorem}. 

$F^\perp$ is isomorphic to $(\SheafEnd(E)/F)^*$ and is thus $\omega$-slope semistable of $\omega$-degree $0$.
The $\omega$-slope of the subsheaf $F+F^\perp$ of $\SheafEnd(E)$ is non-positive, since $\SheafEnd(E)$
is $\omega$-slope semistable. 
We have the short exact sequence
\[
0\rightarrow A \rightarrow F\oplus F^\perp \rightarrow F+F^\perp\rightarrow 0.
\]
Thus, the $\omega$-slope of $A$ is non-negative (and is hence zero), provided $A$ does not vanish. 
$A$ can not vanish, since otherwise $F\oplus F^\perp$ embeds as a subsheaf of $\SheafEnd(E)$ contradicting the maximality of $F$, since $F^\perp$ contains some non-zero polystable subsheaf of $\omega$-slope zero. We have established that $A$ is 
a non-zero nilpotent subsheaf of $\SheafEnd(E)$ of $\omega$-slope zero. If $A$ is $\omega$-slope unstable, replace it by an $\omega$-slope-stable subsheaf of $A$ of maximal $\omega$-slope.
\end{proof}

\begin{lem}
\label{lemma-slope-stable-implies-nilpotents-subsheaves-have-negative-slope}
Let $(X,\omega)$ be a compact K\"{a}hler manifold and $E$ an $\omega$-slope-stable  reflexive $\theta$-twisted  coherent sheaf, for some class $\theta\in H^2_{an}(X,\StructureSheaf{}^*)$.
Then every nilpotent subsheaf of 
$\SheafEnd(E)$ has negative $\omega$-slope.
\end{lem}

\begin{proof}
Let $A\subset \SheafEnd(E)$ be a non-zero nilpotent subsheaf of maximal $\omega$-slope. 
The proof is by contradiction. Assume that $\mu_\omega(A)\geq 0$. 
We may assume that $A$ is a saturated subsheaf, since otherwise the slope of its saturation is larger than or equal to the slope of $A$. 
We may assume that $A$ is $\omega$-slope-stable, possibly after replacing it with a slope-stable subsheaf of maximal $\omega$-slope. 
Let $F\subset E$ be the kernel of the natural homomorphism 
$E\rightarrow \SheafHom(A,E)$. Each stalk of $F$ is the intersection of the kernels of all elements in the corresponding stalk of $A$. 
Let $G$ be the saturation of the image of the natural homomorphism $A\otimes E\rightarrow E$. 
The subsheaves $F$ and $G$ are non-zero subsheaves of $E$ of lower rank, since $A$ is a nilpotent subsheaf.

Assume first that the sheaf $\SheafEnd(G)$ is $\omega$-slope-semistable. The sheaf 
$\SheafHom(A,\SheafEnd(G))$ is $\omega$-slope-semistable of the same non-positive slope as $A^*$,
by Lemma \ref{lemma-polystability-and-tensor-products},
as the sheaves $A$ and $\SheafEnd(G)$ are untwisted and $\omega$-slope-semistable. 
The sheaf $\SheafHom(G,(E/F))$ has positive $\omega$-slope,
since $E$ is $\omega$-slope-stable. This is seen as follows. Set $r:=\rank(E)$, $g:=\rank(G)$, $f:=\rank(F)$. The equality
\[
\mu_\omega(\SheafHom(F,G))=\mu_\omega(F^*\otimes G\otimes E^*\otimes E)=\mu_{\omega}(\SheafHom(E,G))-\mu_\omega(\SheafHom(E,F))
\]
yields
\[
\deg_\omega(\SheafHom(F,G))=fg\left[\mu_{\omega}(\SheafHom(E,G))-\mu_\omega(\SheafHom(E,F))\right]
\]
and 
\begin{eqnarray*}
\deg_\omega(\SheafHom((E/F),G))&=& \deg_\omega(\SheafHom(E,G))-\deg_\omega(\SheafHom(F,G))
\\
&=&g(r-f)\mu_\omega(\SheafHom(E,G))+fg\mu_\omega(\SheafHom(E,F))<0.
\end{eqnarray*}
The natural homomorphism 
$\eta:E/F\rightarrow \SheafHom(A,G)$ is injective, by definition of $F$. Hence, the homomorphism
\[
\eta_* : \SheafHom(G,E/F)\rightarrow \SheafHom(G,\SheafHom(A,G))\cong
\SheafHom(A,\SheafEnd(G))
\]
is injective. This contradicts the semi-stability of $\SheafHom(A,\SheafEnd(G))$.

It remains to consider the case where $\SheafEnd(G)$ is not $\omega$-slope-semistable. 
In this case there exists an $\omega$-slope-stable non-zero nilpotent subsheaf $B\subset \SheafEnd(G)$ of positive $\omega$-slope,
by Lemma \ref{lem-non-semistable-has-nilpotent-subsheaf-of-positive-slope}. 
The composition
\[
B\otimes A \rightarrow \SheafEnd(G)\otimes \SheafHom(E,G)\rightarrow \SheafHom(E,G)\subset \SheafEnd(E)
\]
is a non-zero homomorphism, by definition of $G$. Indeed, 
each stalk of $G$ is the saturation of the sum of images of all elements in the corresponding stalk of $A$. 
The sheaf $(B\otimes A)_{fr}$ is $\omega$-slope-polystable of slope $\mu_\omega(B)+\mu_\omega(A).$
Hence, the image $C$ of the composition displayed above is a  subsheaf, whose slope is strictly larger than that of $A$. 
In particular, the slope of $C$ is positive and $\SheafEnd(E)$ is not $\omega$-slope-semistable. 
This contradicts the maximality of the slope of $A$ among all subsheaves (not necessarily nilpotent) of $\SheafEnd(E)$, by
Lemma \ref{lem-non-semistable-has-nilpotent-subsheaf-of-positive-slope}.
\end{proof}

\begin{proof}[{\bf Proof of Proposition \ref{prop-polystability-of-a-very-twisted-sheaf}}]
The proposition follows immediately from Lemmas \ref{lemma-non-polystable-has-nilpotent-subsheaf-of-zero-slope} and
\ref{lemma-slope-stable-implies-nilpotents-subsheaves-have-negative-slope}.
\end{proof}

{\bf Acknowledgements:}
I would like to thank Misha Verbitsky and Daniel Huybrechts for 
valuable comments. I thank Andrei Caldararu for introducing me to
the theory of twisted coherent sheaves. I thank Jun Li for 
explaining to me the conjectural generalization of the 
Uhlenbeck-Yau Theorem for slope-stable twisted sheaves.  I thank 
Francois Charles for explaining to me his interesting results in 
\cite{charles}. I thank Sukhendu Mehrotra for his careful reading of the appendix. 
I am grateful to the two referees for their detailed and insightful remarks and corrections.
This paper was presented in several workshops. The first two were: 
``Workshop on Moduli spaces of vector bundles'',  
at the Clay Math. Inst., October 2006, 
and ``Non-linear integral transforms: Fourier-Mukai and Nahm'' 
at the Centre de Research Mathematique, Montreal, August 2007.


\end{document}